\documentclass[11pt, twoside, leqno]{amsart}                                                                                                 
\makeatletter
\@namedef{subjclassname@2020}{\textup{2020} Mathematics Subject Classification}
\makeatother

\usepackage{lipsum}
\usepackage{amsfonts}
\usepackage{graphicx}
\usepackage{epstopdf}
\usepackage{algorithmic}
\usepackage{calligra}
\usepackage[dvipsnames]{xcolor}
\usepackage{amsfonts,amsmath,amsthm,amssymb}
\usepackage{mathtools}
 	\usepackage[colorlinks=true,linkcolor=blue, citecolor = blue]{hyperref}
\usepackage[makeroom]{cancel}
\usepackage{autonum}
\usepackage{hhline}
\usepackage{array}
\usepackage{diagbox}
\usepackage{tcolorbox}
\usepackage{mdframed}
\usepackage{multicol}
\usepackage{graphicx}
\usepackage{subcaption}
\usepackage{moreverb}
\usepackage{bbm}
\usepackage[margin=1.25in]{geometry}
\usepackage{todonotes}
\usepackage{scalerel,amssymb}
\allowdisplaybreaks
\usepackage{mathrsfs}  
\usepackage{lineno}
\usepackage{todonotes}
\usepackage{tikz}
\usepackage{appendix}
\usepackage{enumitem}
\usepackage{pgfplots}
\usetikzlibrary{arrows.meta}
\usepackage{siunitx}
\usepackage[numbers,sort&compress]{natbib}
\definecolor{mygreen}{HTML}{43a047}
\usepackage{subcaption}
\usepackage{doi}
\usepackage{soul}
\usepackage[misc]{ifsym}
\usepackage{siunitx}
 	\definecolor{darkgreen}{rgb}{0.0, 0.2, 0.13}

\newcommand{\calT}{{\mathcal{T}}}

\newcommand{\Om}{\Omega}
\newcommand{\D}{\Delta}

\newcommand{\Gn}{\Gamma}

\newcommand{\ut}{u_t}
\newcommand{\utt}{u_{tt}}
\newcommand{\pt}{p_t}
\newcommand{\ptt}{p_{tt}}

\newcommand{\ddt}{\frac{\textup{d}}{\textup{d}t}}

\newcommand{\dt}{\, \textup{d} t}
\newcommand{\ds}{\, \textup{d} s }
\newcommand{\dx}{\, \textup{d} x}
\newcommand{\dG}{\, \textup{d} \Gamma}
\newcommand{\dGs}{\, \textup{d} \Gamma \textup{d}s}
\newcommand{\dGt  }{\, \textup{d} \Gamma \textup{d}t}
\newcommand{\dxs}{\, \textup{d}x\textup{d}s}
\newcommand{\dxt}{\, \textup{d}x\textup{d}t}
\newcommand{\intTO}{\int_0^T \int_{\Omega}}
\newcommand{\intTG}{\int_0^T \int_{\Gamma}}
\newcommand{\inttG}{\int_0^t \int_{\Gamma}}

\newcommand{\intt}{\int_0^t}
\newcommand{\intT}{\int_0^T}
\newcommand{\intO}{\int_{\Omega}}

\newcommand{\nLtwo}[1]{\|#1\|_{L^2(\Omega)}}

\newcommand{\R}{\mathbb{R}} 
\newcommand{\N}{\mathbb{N}} 
\newcommand{\Htwo}{H^2(\Omega)}
\newcommand{\Ltwo}{L^2(\Omega)}
\newcommand{\Linf}{L^\infty(\Omega)}

\newcommand{\Lfour}{L^4(\Omega)}
\newcommand{\Hone}{H^1(\Omega)}

\newcommand{\Hthree}{H^3(\Omega)}

\newcommand{\LtwoLtwo}{L^2(0,T; L^2(\Omega))}
\newcommand{\LinfLinf}{L^\infty(L^\infty(\Omega))}

\def\LtwoLtwo{L^2(L^2(\Omega))}
\def\LinfLtwo{L^\infty(L^2(\Omega))}

\def\Lthree{L^3(\Omega)}

\def\Linf{L^\infty(\Omega)}

\newcommand{\pstar}{p^*}
\newcommand{\pstart}{\pstar_{t}}

\newtheorem{theorem}{Theorem}
\newtheorem{lemma}{Lemma}
\newtheorem{proposition}{Proposition}

\newtheorem{remark}{Remark}

\numberwithin{lemma}{section}
\numberwithin{proposition}{section}
\numberwithin{theorem}{section}
\numberwithin{equation}{section}
\makeatletter
\newcommand{\leqnomode}{\tagsleft@true}
\newcommand{\reqnomode}{\tagsleft@false}
\makeatother

\definecolor{grey}{rgb}{0.5,0.5,0.5}

\definecolor{darkgreen}{rgb}{0,0.5,0}

\def\ball{\mathbb{B}_R}
\def\pd{p^{\textup{d}}}

\def\wt{w_t}
\def\wtt{w_{tt}}

\def\calL{\mathcal{L}}

\def\padj{p^{\textup{adj}}}

\def\padjt{\padj_t}
\def\padjtt{\padj_{tt}}

\def\partialn{\frac{\partial}{\partial n}}

\def\tp{\tilde{p}}

\def\tpadj{\tilde{p}^{\textup{adj}}}
\def\tpadjt{\tilde{p}_t^{\textup{adj}}}

\def\tpd{\tilde{p}^{\textup{d}}}

\def\Xp{\mathcal{X}_p}

\def\intG{\int_{\Gamma}}

\def\Xfad{\Xf^{\textup{adm}}}
\def\spacef{\Xf}
\def\spacefad{\Xf^{\textup{adm}}}
\def\controltostate{\mathcal{S}}

\def\gstar{g^*}

\def\eps{\varepsilon}

\def\peps{p^\eps}

\def\pepst{\peps_{t}}
\def\pepstt{\peps_{tt}}

\def\opeps{z^\eps}
\def\zeps{z^\eps}

\def\zepsn{z^{\eps_n}}
\def\zepsnt{z_t^{\eps_n}}

\def\epsn{\eps_n}

\def\HneghalfG{H^{-1/2}(\Gamma)}

\def\frakK{\mathfrak{K}}

\def\LtwoTLtwo{L^2(0,T; \Ltwo)}

\def\LtwoHone{L^2(\Hone)}

\def\LinfHthree{L^\infty(\Hthree)}

\def\Hfour{H^4(\Omega)}

\def\LtwoHtwo{L^2(\Htwo)}
\def\LinfHtwo{L^\infty(\Htwo)}

\def\csq{c^2}

\def\fraka{\mathfrak{a}}
\def\frakm{\mathfrak{a}}
\def\frakn{\mathfrak{n}}
\def\ulfrakm{\underline{\frakm}}
\def\olfrakm{\overline{\frakm}}

\def\ulfraka{\underline{\fraka}}
\def\olfraka{\overline{\fraka}}

\def\Xf{X_f}
\def\fstar{f^*}

\def\frakmt{\frakm_t}
\def\frakat{\fraka_t}
\def\LtwoG{L^2(\Gamma)}

\def\chizero{\chi_{\Omega_0}}

\def\Gt{G_t}
\def\Gtt{G_{tt}}
\def\tildej{\tilde{j}}

\def\Vm{V^{(m)}}
\def\frakl{\mathfrak{l}}
\def\FG{F_G}

\def\Xfrakm{X_\frakm}
\def\Xfraka{X_\fraka}
\def\Xfrakn{X_\frakn}
\def\Xfrakl{X_\frakl}
\def\pref{p^{\textup{ref}}}
\def\Lsix{L^6(\Omega)}
\def\LtwotLtwo{L^2_t(\Ltwo)}

\def\LinfLthree{L^\infty(\Lthree)}
\newcommand{\Gronwall}{Gr\"onwall}
\def\csqhalf{\frac{\csq}{2}}

\def\LtwotHtwo{L^2_t(\Htwo)}
\def\LtwoHtwo{L^2(\Htwo)}

\def\greg{g^{(\nu)}}

\def\Greg{G^{(\nu)}}
\def\gregt{\greg_t}
\def\gregtt{\greg_{tt}}
\def\Gregt{\Greg_t}
\def\Gregtt{\Greg_{tt}}
\def\FGreg{F_{\Greg}}

\def\wreg{w^{(\nu)}}

\def\wregt{\wreg_t}
\def\wregtt{\wreg_{tt}}

\def\wregm{w^{(\nu, m)}}

\def\wregmt{w_t^{(\nu, m)}}
\def\wregmtt{w_{tt}^{(\nu, m)}}

\def\lesssimT{\lesssim_T}

\def\HthreehalfG{H^{3/2}(\Gamma)}
\def\HfivehalfG{H^{5/2}(\Gamma)}
\def\HonehalfG{H^{1/2}(\Gamma)}

\def\gt{g_t}

\def\LtwotHtwo{L^2_t(\Htwo)}
\def\LtwotHone{L^2_t(\Hone)}

\def\LtwoHthree{L^2(\Hthree)}

\def\Rnu{R_\nu}

\def\ptalpha{\partial_t^\alpha}

\def\frakKonenegalpha{\frakK_{1-\alpha}}

\def\LtwoT{L^2(0,T)}
\def\LoneT{L^1(0,T)}

\def\Calpha{C_\alpha}

\newcommand{\dd}[1]{\mathrm{d}#1}

\def\adjointptalpha{\widetilde{\ptalpha}}

\def\Neumannextension{\mathcal{N}^\Delta}

\def\spaceg{\mathcal{X}_g}
\def\spaceglower{\mathcal{X}^{\textup{low}}_g}

\def\spacegad{\spaceg^{\textup{adm}}}

\def\spaceadj{\mathcal{X}_{\padj}}

\def\gm{g^m}
\def\fm{f^m}

\def\pressurem{p^{(m)}}

\def\gconst{L_1}
\def\fconst{L_2}

\def\Wtwothree{W^{2,3}(\Om)}

\def\Woneinf{W^{1,\infty}(\Om)}

\def\Xk{\mathcal{X}_k}

\def\phim{\varphi^{(m)}}

\def\gdagger{g^\dagger}
\def\fdagger{f^\dagger}

\def\tildeg{\tilde{g}}
\def\tildef{\tilde{f}}

\def\pressuremd{p^{\textup{d}, m}}
\def\pressurend{p^{\textup{d}, n}}
\def\gammam{\gamma_m}
\def\alpham{\gamma_m}

\def\gn{g^n}
\def\fn{f^n}

\def\pdalpham{\pd_{\alpham}}
\def\pdgamma{p^{\textup{d}}_{\gamma}}
\def\pdgammam{p^{\textup{d}}_{\gammam}}

\def\ggamma{g_\gamma}
\def\fgamma{f_\gamma}

\def\ggammam{g_{\gammam}}
\def\fgammam{f_{\gammam}}
\def\Halpha{H^\alpha}
\def\HalphaT{H^\alpha(0,T)}

\def\Jalpha{\mathcal{J}^\alpha}

\def\Hsubalpha{H_\alpha}
\def\ulC{\underline{C}}
\def\HsubalphaT{H_\alpha(0,T)}

\def\Jnegalpha{\mathcal{J}^{-\alpha}}

\def\calR{\mathcal{R}}

\def\pstarone{p^{*, (1)}}
\def\pstartwo{p^{*, (2)}}
\def\pone{p^{(1)}}
\def\ptwo{p^{(2)}}
\def\pstardiff{\bar{p}^*}
\def\pdiff{\bar{p}}

\def\lambdam{ \lambda^{(m)}}

\def\Xlow{X_{\textup{low}}}

\def\uGalerkin{u^{(m)}}
\def\vGalerkin{v^{(m)}}
\def\utGalerkin{u_{t}^{(m)}}
\def\uttGalerkin{u_{tt}^{(m)}}

\def\uonem{u_1^{(m)}}

\def\lhs{\textup{lhs}}
\def\rhs{\textup{rhs}}

\def\etaeps{\eta^\eps}

\def\geps{g^\eps}
\def\gepstilde{\bar{g}^\eps}
\def\gtilde{\tilde{g}}
\def\calLfraka{\calL_{\fraka, \frakn, \frakl}}                                                                                                                                                                                      
\makeatletter                                                                                                                                                                                                                
\@namedef{subjclassname@2020}{\textup{2020} Mathematics Subject Classification}      
\makeatother                                                                                                                                                                                                                                   
                                                                                                        
\title[Optimal control of the Westervelt equation with fractional attenuation]{Optimal Neumann boundary and distributed control of the Westervelt equation with time-fractional attenuation}               
\subjclass[2020]{49J20, 35L70, 35R11}                                                          
                                              
\keywords{Westervelt equation, fractional PDEs, distributed optimal control, boundary optimal control, optimality conditions}                   
\author{Vanja Nikoli\'c$^\dagger$}  
\thanks{$^\dagger$Department of Mathematics,
	Radboud University,      
	Heyendaalseweg 135,     
	6525 AJ Nijmegen, The Netherlands (\href{vanja.nikolic@ru.nl}{vanja.nikolic@ru.nl})}   
\author{Belkacem Said-Houari$^\ddag$}
\thanks{$^\ddag$Department of Mathematics, College of Sciences, University of
	Sharjah, P.\ O.\ Box: 27272, Sharjah, United Arab Emirates (\href{bhouari@sharjah.ac.ae}{bhouari@sharjah.ac.ae})}
\begin{document}
\vspace*{4mm}  
\begin{abstract}
Optimal control of nonlinear acoustic waves is relevant in many medical ultrasound technologies, ranging from cancer therapy to targeted drug delivery, where it can help guide the precise deposition of acoustic energy. In this work, we study Neumann boundary and distributed control problems for tracking a prescribed pressure field governed by the Westervelt equation with time-fractional dissipation. This model captures nonlinear ultrasonic wave propagation in biological media and accounts for the experimentally observed power-law attenuation. We begin by extending the existing well-posedness theory for time-fractional equations to include inhomogeneous Neumann boundary data used as control inputs, which requires constructing an appropriate data extension and regularization. Using these analytical results for the forward problem, we prove the existence of globally optimal controls and analyze the stability of the optimization problem with respect to perturbations in the target pressure field and to vanishing regularization parameters. Finally, we investigate the associated adjoint equation, which has state-dependent coefficients, and use it to derive first-order necessary optimality conditions.

\end{abstract}           
	\vspace*{-7mm}         
	\maketitle                                                                   
\section{Introduction}
In medical acoustics, having precise control of ultrasonic waves in the region under treatment is often needed for both safety and efficiency; the applications include tissue ablation in cancer treatments~\cite{kennedy2003high}, lithotripsy~\cite{ikeda2016focused}, and ultrasound-enhanced drug delivery~\cite{postema2007ultrasound}. This motivates the present study of an optimal control problem subject to Westervelt's wave model of nonlinear acoustic propagation through biological media. In such media, ultrasound attenuation exhibits a non-integer power-law dependence on frequency, which can be accurately captured through time-fractional dissipation terms; see~\cite{parker2022power, wells1975absorption, hill1978ultrasonic, holm2019waves} for modeling details. \\
\indent Control of the acoustic waves can be achieved through the boundary or by having a distributed acoustic source. We consider both settings and treat them in a largely unified theoretical framework. Given the desired acoustic pressure $\pd \in C([0,T]; \Ltwo)$, we consider a regularized objective 
\begin{equation}\label{Object_fun}
\begin{aligned} 
	J(p, g, f) =&\, \begin{multlined}[t] \frac{\nu}{2}\| p-\pd\|^2_{L^2(0,T; L^2(\Omega_0))} +\frac{1-\nu}{2}\| p(T)-\pd(T)\|^2_{L^2(\Omega_0)}  
		+\calR(g,f),
	\end{multlined}
\end{aligned}
\end{equation}
with $\nu \in \{0,1\}$. We require that the pressure is close to the desired pressure locally on a region of interest $\Omega_0 \subset \Omega$ and possibly only at final time $T$ (if $\nu=0$).  The regularizing functional is given by
\begin{equation} \label{def calR}
	\calR(g,f) =	\frac{\gamma}{2}\|g\|^2_{L^2(0,T; L^2(\partial \Omega))} +\frac{\eta}{2}\| f\|^2_{L^2(0,T; L^2(\Omega))}, \quad \gamma, \eta \geq 0.
\end{equation} The acoustic pressure field is obtained by solving the Neumann initial boundary-value problem for the Westervelt wave equation~\cite{westervelt1963parametric, prieur2011nonlinear}:
\begin{equation} \label{ibvp West} \tag{IBVP$_\text{West}$} 
	\left\{
 \begin{aligned}
		&p_{tt}-c^2 \Delta p - b \Delta \ptalpha p = k(x)\left(p^2\right)_{tt}+f, \qquad &&\text{in} \ \Omega \times (0,T), \\
		&\frac{\partial p}{\partial n}=\, g \quad &&\text{on} \ \Gn = \partial \Omega, \\[1mm]
		&(p, p_t)\vert_{t=0}= (0, 0),
	\end{aligned} 
	\right.
\end{equation}
where the source term $f$ and the Neumann boundary data $g$ act as controls; the latter acts on the boundary $\Gamma= \partial \Omega$.  In \eqref{ibvp West}, $p=p(x,t)$ denotes the acoustic pressure,  $c>0$ is the speed of sound in the medium, $b >0$ the attenuation coefficient, and $k =k(x)$ is the nonlinearity coefficient. We allow for $k$ to vary in space to model varying influences of nonlinearity closer and farther from the source; the size of $k$, which will later be assumed to be small in a suitable norm, also regulates the well-posedness of \eqref{ibvp West}, as seen in Theorem~\ref{thm: wellp forward} below. The acoustic dissipation term $-b \Delta \ptalpha p$ in \eqref{ibvp West} involves the  Djrbashian--Caputo fractional derivative operator $\ptalpha(\cdot)$ of order $\alpha \in (0,1)$ to accurately model experimentally observed power-law attenuation in tissue media; we provide its definition and further for this work relevant theoretical details in Section~\ref{sec: Preliminaries}. Given the technical intricacies of studying optimal control problems involving nonlinear wave propagation, we focus on boundary and distributed excitation with homogeneous initial conditions. \\[2mm]
\noindent \textbf{Main contributions}. The aim of this work is to deepen the theoretical understanding of optimal control of nonlinear acoustic waves in complex propagation media.   We investigate the following optimal control problem:
\begin{equation} \label{opt problem} \tag{P}
	\begin{aligned}
	& \inf_{(p, g, f) \in M} J(p, g, f) \\
\text{with}\qquad& \\
	&	M= \left\{ (p, g, f) \in \Xp  \times \spacegad \times \Xfad:\, p \ \text{solves }\ \eqref{ibvp West} \right\},
	\end{aligned}
\end{equation}
where the space $\Xp$ to which the pressure field belongs will be made precise in the well-posedness analysis in Section~\ref{sec: wellp forward}. The admissible space $\spacegad$ of boundary controls is a subset of a closed ball in a suitable Banach space where the compatibility conditions with homogeneous initial data hold, whereas the admissible space $\spacefad$ of distributed controls is a closed ball in a Hilbert space; we refer to Section~\ref{sec: existence opt control} for details. These choices are dictated by the needs of the well-posedness analysis of \eqref{ibvp West}. The main contributions of our work pertain to 
\begin{enumerate}[label=(\roman*), itemsep=0.15cm]
	\item showing the existence of globally optimal controls for \eqref{opt problem},
		\item ensuring stability with respect to perturbations in the desired pressure $\pd$ and the vanishing regularization limit, and 
		\item  deriving adjoint-based first-order optimality conditions for the distributed and Neumann control cases. \vspace*{1mm}
\end{enumerate}
To the best of our knowledge, this is the first optimal control analysis for the Westervelt equation with time-fractional attenuation and Neumann boundary control. The principal difficulty in obtaining (i)--(iii), beyond adapting established optimal control frameworks~\cite{troltzsch2024optimal, manzoni2021optimal, engl1989convergence, hinze2008optimization}, stems from the quasilinear nature of the state equation, which can be rewritten as
\begin{equation} \label{West rewritten}
	((1-2kp)\pt)_t-c^2 \Delta p - b \Delta \ptalpha p = f.
\end{equation}
A rigorous analysis of \eqref{West rewritten} must guarantee that $1-2kp$ stays strictly positive to avoid degeneracy, which requires a bound on $\|k u\|_{L^\infty(0,T; L^\infty(\Omega))}$.  Because the fractional dissipation is weaker than classical strong damping, achieving such bounds necessitates higher-order energy estimates carried out on a linearized problem in combination with a fixed-point argument. The presence of a Neumann boundary control makes this analysis particularly intricate since the well-posedness theory requires a suitable extension of regularized boundary data into the interior of the domain. Furthermore, in the adjoint problem,  which has state-dependent coefficients, the regularity of data is limited as a consequence of the localization of the tracking functional to $\Omega_0$, so the higher-order smoothness arguments used in the forward analysis cannot be transferred to it and we have to also devise well-posedness arguments in a lower-regularity regime.\\[2mm]
\noindent \textbf{Related work}. Research on fractional evolution equations is very active, including in the neighboring area of inverse problems; we refer to the books~\cite{kaltenbacher2023inverse, jin2023numerical, jin2021fractional} and the references contained therein. In closer relation to the present topic, we point out the work on the identification of the nonlinearity parameter in the fractionally attenuated Westervelt equation with Dirichlet data in~\cite{kaltenbacher2022inverse}. Optimal control of fractional evolution equations is likewise an increasingly active area of investigations, although most works have focused on linear or non-wave problems; see, for example,~\cite{mophou2011optimal, antil2016space, jin2020pointwise, gunzburger2019error}. Analysis and numerical analysis of a distributed optimal control problem for a linear wave model with fractional attenuation in the form of $\ptalpha u$ and homogeneous Dirichlet data can be found in~\cite{wang2024finite}. \\
  \indent Boundary optimal control of the strongly damped Westervelt equation (that is, with the dissipation term $- b \Delta \pt$ in place of $- b \Delta \ptalpha p$ in \eqref{ibvp West}) has been rigorously investigated in~\cite{clason2009boundary}. In this strongly damped case, the equation is known to have parabolic-like character which can be exploited to ensure its global-in-time well-posedness in different settings in terms of boundary conditions; see, e.g.,~\cite{kaltenbacher2009global, meyer2011optimal}. With time-fractional attenuation as in \eqref{ibvp West}, this property is lost and the energy arguments do not carry over. The well-posedness of the fractionally damped Westervelt equation supplemented by Dirichlet boundary data has been established recently in~\cite{kaltenbacher2022inverse}; see also \cite{baker2024numerical, kaltenbacher2022limiting}. 
   Well-posedness in the case of Neumann data as in \eqref{ibvp West} is, to the best of our knowledge, still an open problem and it is thus the first step we take in this work toward understanding \eqref{opt problem}.%
\subsection*{Organization of the paper} 
The remainder of the paper is structured as follows. Section \ref{sec: Preliminaries} summarizes the necessary background on fractional calculus based on the Djrbashian–Caputo derivative, which we use in the analysis of the optimal control problem. Section \ref{sec: wellp forward} establishes well-posedness for a linearized wave equation in both low- and high-regularity settings. The low-regularity result is essential for the adjoint analysis and for proving differentiability of the control-to-state map. The high-regularity result, which requires the construction of a suitable extension of regularized Neumann data, is then combined with Banach’s fixed-point theorem to prove the well-posedness of the state problem when $k=k(x)$ is sufficiently small. This result in contained in Theorem~\ref{thm: wellp forward}. In Section \ref{sec: existence opt control} we prove the existence of optimal controls in Theorem~\ref{thm: existence opt control} and then demonstrate stability of the problem with respect to perturbations in the target pressure, as well as in the limit of vanishing regularization parameters $\eta=\gamma \searrow 0$. Finally, building upon these results, in Section~\ref{sec: optimality conditions} we analyze the adjoint problem, establish the differentiability of the control-to-state operator, and then derive the necessary optimality conditions for the Neumann $(J=J(p, g))$ and  distributed $(J=J(p,f))$ optimal control problems separately. The optimality conditions are stated in Theorems~\ref{thm: opt cond boundary control} and~\ref{thm: opt cond distributed control}, respectively.
\\[2mm]
\noindent \textbf{Notation}. We use $(\cdot, \cdot)_{\Ltwo}$ to denote the $\Ltwo$ inner product. When working with Bochner spaces  $W^{m,p}(0,T;X)$, we often omit the time interval $(0,T)$ from the notation of norms. For example, $\|\cdot\|_{L^2(\Ltwo)}$ denotes the norm in $L^2(0,T; \Ltwo)$. A subscript $t$ indicates that the temporal domain is $(0,t)$ for some $t \in [0,T]$. For example, $\|\cdot\|_{L^2_t(\Ltwo)}$ denotes the norm in $L^2(0,t; \Ltwo)$ for $t \in (0,T)$. \\
\indent We frequently use $\lhs \lesssim \rhs$ in the estimates to denote $\lhs \leq C \rhs$, where $C>0$ is a generic constant. When the hidden constant depends on the final time $T$, we use the notation $\lhs \lesssimT \rhs$.
\section{Theoretical preliminaries in fractional calculus} \label{sec: Preliminaries}
In this section, we recall several results from fractional calculus involving the Djrbashian--Caputo fractional derivative, which will be needed for the subsequent analysis of the state and optimal control problems.  Most results stated below can be found in~\cite[Ch.\ 2]{kubicaTimefractionalDifferentialEquations2020}, where additional details are provided. \\[2mm]
\noindent \textbf{Fractional calculus}. 
For $\alpha>0$,  the Riemann--Liouville fractional integral operator is given by 
	\begin{equation}\label{def fract int}
		\Jalpha v (t) = \frac{1}{\Gamma(\alpha)} \int_0^t (t-s)^{\alpha -1} v(s) \ds \quad \text{with the domain } \mathcal{D}(\Jalpha)=\LtwoT,
	\end{equation}
	where $\Gamma(\cdot)$ denotes the Gamma function; see~\cite{kubicaTimefractionalDifferentialEquations2020, gorenflo1999operator}. For $0<\alpha<1$, the fractional Sobolev space  $\HalphaT$ is endowed with the norm 
\begin{equation}
	\|v\|_{H^{\alpha}(0,T)} =
	\left(
	\|v\|_{L^{2}(0,T)}^{2}
	+ \int_{0}^{T} \int_{0}^{T} 
	\frac{|v(t) - v(s)|^{2}}{|t-s|^{1+2\alpha}} \dt \ds
	\right)^{\tfrac{1}{2}};
\end{equation}	
cf.~\cite{di2012hitchhiker}. According to \cite[Theorem 2.1]{kubicaTimefractionalDifferentialEquations2020},  the fractional integral operator $\Jalpha: L^2(0,T) \rightarrow \HsubalphaT$ is bijective, where the space $\HsubalphaT$ is defined as
	\begin{equation} \label{def Hsubalpha}
		\begin{aligned}
			\Hsubalpha(0,T) = \begin{cases}
				\left \{ v \in \HalphaT: \ v(0) =0\right \} \quad &\text{if} \ \frac12< \alpha <1, \\[1mm]
				\{v \in H^{1/2}(0,T): \ \int_0^T \frac{|v(t)|^2}{t}\dt <\infty\}\quad &\text{if} \  \alpha =\frac12, \\[1mm]
				\HalphaT &\text{if} \quad 0<\alpha < \frac12,
			\end{cases}
		\end{aligned}
	\end{equation}
	endowed with the norm
	\begin{equation} \label{norm Hsubalpha}
		\begin{aligned}
		\|v\|_{\HsubalphaT} = \begin{cases}
			\|v\|_{\HalphaT}, \quad &0< \alpha <1, \ \alpha \neq \frac12, \\
			\left(\|v\|^2_{H^{1/2}(0,T)}+  \int_0^T \frac{|v(t)|^2}{t}\dt\right)^{1/2}, \quad &\alpha = \frac12.
		\end{cases}
		\end{aligned}
	\end{equation}
It is known that the space ${_0}C^1[0,T] = \{v \in C^1[0,T]: v(0)=0\}$ is dense in $\Hsubalpha(0,T)$:
	\begin{equation} \label{density}
		\overline{{_0}C^1[0,T]}^{\Hsubalpha(0,T)} = \Hsubalpha(0,T);
	\end{equation}
see~\cite[Lemma 2.2]{kubicaTimefractionalDifferentialEquations2020}. The (generalized) Djrbashian--Caputo fractional derivative is defined as the inverse of the fractional integral operator in \eqref{def fract int}, that is,
	\begin{equation} \label{def Caputo regularized}
		\ptalpha v = \Jnegalpha v, \quad   \mathcal{D}(\ptalpha)= \Hsubalpha(0,T);
	\end{equation}
	see~\cite[Definition 2.1]{kubicaTimefractionalDifferentialEquations2020}.  On account of~\cite[Theorem 2.4, (2.25)]{kubicaTimefractionalDifferentialEquations2020}, for $\alpha \in (0,1)$, $\ptalpha$ is an isomorphism between $\HsubalphaT$ and $\LtwoT$, and thus we have 
	\begin{equation}
		\| \ptalpha v\|_{\LtwoT} \sim \|v\|_{\HsubalphaT}.
	\end{equation}
		If $v \in {_0W}^{1,1}(0,T)= \{ v \in W^{1,1}(0,T): v(0)=0\}$, then the generalized Djrbashian--Caputo matches the usual definition
	\begin{equation}\label{def Caputo}
		\ptalpha v =  \frac{1}{\Gamma(1-\alpha)} \int_0^t (t-s)^{-\alpha} v_t(s)\ds, \ \quad \alpha \in (0,1);
	\end{equation}
	see~\cite[Theorem 2.4]{kubicaTimefractionalDifferentialEquations2020}.  In this case, $\ptalpha v = \frakKonenegalpha * v_t $
	with the weakly singular memory kernel
	\begin{equation} \label{Abel kernel}
		\frakKonenegalpha(t) = \frac{1}{\Gamma(1-\alpha)}t^{-\alpha}.
	\end{equation} 
Moreover, $\frakKonenegalpha\in L^p(0,T)$ for $0<\alpha<1/p$ and $p \in (1, \infty)$.\\[2mm]
\noindent \textbf{Helpful inequalities}.  Given a Banach space X, we introduce
	\begin{equation}
		\Hsubalpha(0,T; X) = \Jalpha L^2(0,T; X),
	\end{equation}
	which is a Banach space endowed with the norm $\|u\|_{\Hsubalpha(0,T; X)} = \|\ptalpha u\|_{L^2(0,T; X)}$; see~\cite[Lemma 2]{huang2025well}.
 For $u \in \Hsubalpha(0,t; \Ltwo)$,  it holds that  
	\begin{equation} \label{coercivity ineq}
	\int_0^t \int_\Omega \partial_t^\alpha u \,  \partial_t u
\dxs\geq \frac{1}{2}(\frakK_\alpha* \|\partial_t^\alpha u\|_{\Ltwo}^2)(t)\geq \Calpha(T) \|\partial_t^\alpha u\|_{\LtwotLtwo}^2,
	\end{equation}
	where $\Calpha(T) = \frac{T^{\alpha}}{2\Gamma(1-\alpha)}$, by~\cite[Theorem 3.1]{kubicaTimefractionalDifferentialEquations2020}. In the analysis, we  use the $\alpha$-uniform lower bound $\Calpha(T) \geq \ulC(T)$ to guarantee that the constants in the estimates are $\alpha$-independent.
We also use the inequality
	\begin{equation} \label{coercivity ineq 2} 
		\int_0^t \int_\Omega u(s) \left(\Jalpha u\right)(s)    \dxs \geq 0, \quad u \in L^2(0,t; \Ltwo),
	\end{equation}
see, e.g.,~\cite[Sec.\ 5]{kaltenbacher2022limiting} and \cite[Lemma 2]{gorenflo1999operator}.	Furthermore, for any $u \in \Hsubalpha(0,T;X)$, the identity $u= \Jalpha \Jnegalpha u$, together with
 Young's convolution inequality, yields 
	\begin{equation} \label{fundamental inequality}
		\begin{aligned}
		\|u\|_{L^2(0,T; X)} \leq \|\frakK_\alpha\|_{\LoneT} \|\ptalpha u\|_{L^2(0,T; X)}.
		\end{aligned}
	\end{equation}
	
\noindent \textbf{Adjoint of the fractional derivative}. 	
In the derivation of first-order optimality conditions via the adjoint problem corresponding to~\eqref{ibvp West}, we require the Banach space adjoint $\adjointptalpha$ of the time-fractional derivative operator $\partial_t^{\alpha}$:
	\begin{equation} \label{adjoint operator rule}
	\begin{aligned}
		\int_0^T (\partial_t^{\alpha} v)(t) \varphi(t) \dd{t} 
		= \int_0^T v(t) ( \adjointptalpha \varphi)(t) \dt, \quad v \in \Hsubalpha(0,T). 
	\end{aligned}
\end{equation}
The adjoint operator is defined as
	\begin{equation}
		\adjointptalpha = \widetilde{\Jalpha}^{-1} \quad \text{with the domain}\  \mathcal{D}(\adjointptalpha) = \tau \Hsubalpha(0,T), \ \text{for } \alpha>0, 
	\end{equation}
	where $(\tau v)(t) \coloneqq v(T-t)$ denotes the time reversal operator and
	\begin{equation} \label{id widetilde Jalpha}
		\widetilde{\Jalpha}v(t) = \frac{1}{\Gamma(\alpha)}\int_t^T(s-t)^{\alpha-1}v(s)\ds = \tau \Jalpha(\tau v)(t);
	\end{equation}
see~\cite[Theorem I-5]{Yamamoto2022Adjoint}. To carry out the adjoint-based arguments in Section~\ref{sec: existence opt control}, we also require the following time-reversal identity:
	\begin{equation} \label{identity timeflip}
	\begin{aligned}
		\tau \left( \adjointptalpha \varphi \right)  = \ptalpha \left(\tau \varphi \right), \qquad \varphi \in \tau \Halpha(0,T).
	\end{aligned}
	\end{equation}
Indeed, by \eqref{id widetilde Jalpha}, we have
	\begin{equation}
	\begin{aligned}
\tau	\Jalpha \left(	\tau  \left( (\widetilde{\Jalpha})^{-1} \varphi \right)  \right) = \widetilde{\Jalpha}\left(\widetilde{\Jalpha}\right)^{-1} \varphi = \varphi.
	\end{aligned}
\end{equation}
Hence, since $(\tau(\tau v)(t))=v(t)$, we have 
\begin{equation}
	\Jalpha \left(	\tau  \left( (\widetilde{\Jalpha})^{-1} \varphi \right)  \right) = \tau  \varphi.
\end{equation}
Then because $\Jnegalpha = (\Jalpha)^{-1}$ on $\Hsubalpha(0,T)$, it follows that
\begin{equation}
	\tau  \left( (\widetilde{\Jalpha})^{-1} \varphi \right)= \Jnegalpha(\tau \varphi),
\end{equation}
which is precisely \eqref{identity timeflip}.
\\ \indent For $v$, $\varphi \in H^1(0,T)$, with $v(0)=\varphi(T)=0$,  the adjoint operator can be expressed as
	\begin{equation} \label{adjoint operators}
		\begin{aligned}
			(\adjointptalpha \varphi)(s) &= -\frac{1}{\Gamma(1-\alpha)} \int_s^T (t-s)^{-\alpha} \varphi'(t) \dd{t}; 
		\end{aligned}
	\end{equation}
see~\cite[Sec.\ 3]{kaltenbacher2022fractional}.                                       
\section{Analysis of wave equations with time-fractional attenuation} \label{sec: wellp forward} 
In this section, we lay the groundwork for the analysis of the optimal control problem by studying the well-posedness of Neumann boundary value problems for both the linearized wave equation and the Westervelt equation in~\eqref{ibvp West}.  We consider a linearized wave equation of the form
\begin{equation}\label{linearized eq} 
\begin{aligned}
		&\ \fraka(x,t) \utt - \csq \Delta u-b \Delta \ptalpha u+ \frakl(x,t)  \ut+ \frakn(x,t) u= F(x,t)
	\end{aligned} 
\end{equation}
The well-posedness result for \eqref{linearized eq} is relevant not only in the fixed-point argument for the forward problem, but also in the analysis of the adjoint problem and in establishing differentiability of the control-to-state mapping. We conduct the linear acoustic investigations under the following non-degeneracy assumption on the leading variable coefficient: we assume that there exist $\ulfrakm$, $\olfrakm$,  such that
\begin{equation}\label{nondegeneracy assumption frakm frakm}
	0<\ulfraka \leq \fraka(x,t)\leq \olfraka \quad \text{for all $(x,t)\in \Omega  \times [0,T]$}.
\end{equation}
In terms of later studying the nonlinear pressure problem, this coefficient can be seen as a placeholder for $1 -2k p$, whose non-degeneracy will be ensured through smallness of the nonlinearity coefficient  $k=k(x    )$ in a suitable norm in the fixed-point argument; cf.~\eqref{West rewritten}. The coefficient $\frakl$ can be seen as a placeholder for $2k \ut$. The term 
$\frakn$ is included to accommodate the analysis of differentiability of the control-to-state mapping Section~\ref{sec: optimality conditions}; cf.~Proposition~\ref{prop: diff controltostate}. \\
\indent Because the adjoint problem and the differentiability analysis will take place in a lower regularity setting than the state equation, we establish two well-posedness results: one in a lower-regularity regime and another in a higher-regularity regime.
	\subsection{Analysis of a linearized wave problem: Lower-regularity regime}
The first well-posedness result is stated in a relatively low regularity setting for the data and variable coefficients, suitable for later analyzing the adjoint problem and the differentiability of the control-to-state mapping. As the adjoint problem contains a non-zero terminal condition (see \eqref{adjoint problem}), arising from the time-localized part of the objective function, we analyze here a linearized wave problem with $\ut \vert_{t=0}=u_1$:
				\begin{equation}\label{linearized problem} \tag{IBVP$_\text{lin}$}
	\left\{	\begin{aligned}
		&\ \fraka(x,t) \utt - \csq \Delta u-b \Delta \ptalpha u+ \frakl(x,t)  \ut+ \frakn(x,t) u= F(x,t),\\
		&\ \frac{\partial u}{\partial n} = g  \quad \text{on } \ \Gamma, \\
		&\ (u, \ut) \vert_{t=0}= (0, u_1).  
	\end{aligned} \right.
\end{equation}					
We define weak solutions of \eqref{linearized problem} as functions satisfying
\begin{equation} \label{weak linearized problem low} 
	\begin{aligned}
	&	\begin{multlined}[t]-	\intTO (\fraka \varphi)_t \, \ut  \dxt +\intTO (\csq \nabla u +b \nabla \ptalpha u)\cdot \nabla \varphi \dxt 
		\\	\hspace*{2cm}+ \intTO  (\frakl  \ut+ \frakn u) \varphi \dxt - \intTG (c^2 g+ b \ptalpha g) \varphi \dGt   
		\end{multlined}\\
					=&\,   \intO  \fraka(0)u_1 \varphi(0) \dx+  \intTO F \varphi \dxt
	\end{aligned}
\end{equation}
 for all test functions   $\varphi \in L^2(0,T; \Hone) \cap H^1(0,T; \Ltwo)$ with $\varphi(T)=0$.
	\begin{proposition}\label{prop: wellp forward lower}
		Let $T>0$ and let $\Omega \subset \R^d$ with $d \in \{1,2,3\}$ be a Lipschitz-regular bounded domain. Let $c>0$,  $b \in (0, \bar{b}]$ for some $\bar{b}>0$, and $\alpha \in (0,1)$. Assume that the variable coefficients in \eqref{linearized problem} satisfy 
	\begin{equation}
		\begin{aligned}
			&\fraka \in    W^{1, \infty}(0,T; \Linf) ,\quad
			 \frakl \in  L^\infty(0,T; \Linf), \quad
			\frakn \in  L^2(0,T; \Lthree),
		\end{aligned}
	\end{equation}
	and that the non-degeneracy condition on $\fraka$ in \eqref{nondegeneracy assumption frakm frakm} holds. 
	Furthermore, let $F \in \LtwoTLtwo$, $g \in H^2(0,T; \HneghalfG)$, and $u_1 \in \Ltwo$.
	Then there exists a solution $u \in \Xlow$  of \eqref{weak linearized problem low}, where
	\begin{equation} \label{def Xlow}
	\begin{aligned}
		\Xlow = \left\{ u \in L^\infty(0,T; \Hone): \ut \in  L^{\infty}(0,T; \Ltwo), \ \ptalpha u \in (0,T; \Hone) \right\}. 
	\end{aligned}
\end{equation}
	The solution satisfies the estimate
	\begin{equation} \label{final est lower} 
		\begin{aligned}
			\|u\|^2_{\Xlow} \lesssim &\, \begin{multlined}[t] \Lambda_0(\fraka, \frakl, \frakn)	 \left(	\nLtwo{ u_1}^2 
				+ \|F\|^2_{\LtwoLtwo}
				+\|g\|^2_{H^2(\HneghalfG)} \right), \end{multlined}
		\end{aligned}
	\end{equation}
	where
	\begin{equation} \label{def Lambda_0}
		\begin{aligned}
		\Lambda_0(\fraka, \frakl, \frakn)	=	\exp\left\{CT(	1+\|\fraka\|_{W^{1, \infty}(\Linf)}+\|\frakl\|_{L^\infty(\Linf)}+\|\frakn\|^2_{L^2(\Lthree)})\right\}.
		\end{aligned}
	\end{equation}
	The hidden constants are independent of $b$ and $\alpha$.
	\end{proposition}
	\begin{proof}
		We establish well-posedness of \eqref{linearized problem} using a Faedo–Galerkin method, analogous to the approach used in the analysis of linear wave equations with time-fractional attenuation and homogeneous Dirichlet boundary conditions; see, e.g.,~\cite{kaltenbacher2022inverse, kaltenbacher2022limiting, careaga2025westervelt}. The main difference lies in the testing procedure, which here is adapted to handle lower regularity. To discretize \eqref{linearized problem homogen} in space, we employ as the basis functions $\{\phim\}_{m \geq 1}$ eigenfunctions of the Laplacian with homogeneous Neumann conditions 
		\begin{equation}\label{bases_Lapl}
		-	\Delta \phim = \lambdam \phim \ \text{ in } \Omega,  \quad \frac{\partial \phim}{\partial n}=0 \ \text{ on } \Gamma.
		\end{equation}
		The problem can then be posed in $\Vm=\textup{span}\{\varphi^{(1)}, \ldots, \phim\}$ and solved for
		\begin{equation}
			\uGalerkin(t) = \sum_{\ell=1}^m \xi_m^\ell(t) \phim,
		\end{equation}
		by requiring that
				\begin{equation} \label{identity_1 Step 1 low}
			\begin{aligned} 
				&	\begin{multlined}[t]
					(\fraka \uttGalerkin ,   \vGalerkin)_{\Ltwo} +	(\csq \nabla \uGalerkin+ b \nabla \ptalpha \uGalerkin,\, \nabla   \vGalerkin)_{\Ltwo} \\ \hspace*{2cm}
					-(\csq g+b \ptalpha g, \vGalerkin )_{L^2(\Gamma)}
					= ( F-\frakl \utGalerkin-\frakn \uGalerkin,  \vGalerkin)_{\Ltwo}, 
				\end{multlined} 
			\end{aligned}
		\end{equation}
	for all $\vGalerkin \in \Vm$,	where we impose $\uGalerkin(0)=0$ and $(\uGalerkin_t(0), \phim)_{\Ltwo} = (u_1, \phim)_{\Ltwo}$, $k=1, \ldots m$ (in other words we choose approximate initial data as $L^2$ projections of the exact data).
		Existence and uniqueness of $\uGalerkin \in H^2(0,T  ; \Vm)$ in this form follow in a directly analogous manner to the Dirichlet case by restating the problem in terms of $\boldsymbol{\mu}=  \boldsymbol{\xi}_{tt}$ with $\boldsymbol{\xi} = [\xi^\ell_1 \ \xi_2^\ell \, \ldots\, \xi^\ell_m]^T$ and relying on the existence theory for Volterra integral equations of second kind; see~\cite[Appendix A]{kaltenbacher2022limiting} for details. We thus omit those arguments here and proceed directly to derive the energy estimate that is uniform with respect to the discretization parameter $m$. \\[1mm]
		
		\noindent $\bullet$ \underline{Testing with $\utGalerkin$}:  Testing \eqref{identity_1 Step 1 low} with $v^{(m)} =\utGalerkin$, and invoking the coercivity bound \eqref{coercivity ineq}, we obtain the inequality
		\begin{equation} \label{energy_ineq_1 Step 1 low_}
			\begin{aligned}
				&\frac12\nLtwo{\sqrt{\fraka(s)}  \utGalerkin (s)}^2  \Big \vert_{0}^t+ \csqhalf\nLtwo{\nabla \uGalerkin(t)}^2 + b
				\Calpha(T) \| \nabla \ptalpha  \uGalerkin \|^2_{\LtwotLtwo}  \\
				\leq&\,\begin{multlined}[t] 
					\frac12 \int_0^t (\frakat \utGalerkin,  \utGalerkin)_{\Ltwo}\ds					+\intt ( \frakl\utGalerkin +\frakn \uGalerkin,  \utGalerkin )_{\Ltwo} \ds.\\
					 + \inttG (c^2 g+ b \ptalpha g) \utGalerkin \dGs+ \intt (F, \utGalerkin)_{\Ltwo} \ds .
				\end{multlined}
			\end{aligned} 
		\end{equation}
		We estimate the first term on the right-hand side using H\"older's inequality:
		\begin{equation}\label{Est_1 low}
			\begin{aligned}
				\frac12 \int_0^t (\frakat \utGalerkin,  \utGalerkin)_{\Ltwo}\ds\lesssim &\,\int_0^t \left\|\frakat\right\|_{\Linf} \nLtwo{  \utGalerkin}^2\ds.
			\end{aligned}
		\end{equation}
		Similarly, using also Young's $\eps$-inequality, and the Sobolev embedding $H^1(\Omega)\hookrightarrow L^6(\Omega)$, we have, for any $\eps>0$, 
		\begin{equation}\label{Est_2 low}
			\begin{aligned}
				& \intt ( \frakl\utGalerkin +\frakn \uGalerkin,  \utGalerkin )_{\Ltwo} \ds\\ 
				\lesssim&\, \int_0^t \left\|\frakl\right\|_{\Linf} \nLtwo{\utGalerkin}^2\ds + \intt \|\frakn\|_{\Lthree}\|\uGalerkin\|_{\Lsix}\|\utGalerkin\|_{\Ltwo} \ds, \\
				\lesssim&\, 
				 \int_0^t \left( \|\frakl\|_{\Linf} +\|\frakn\|^2_{\Lthree}\right) \nLtwo{\utGalerkin}^2\ds +\eps \max_{s \in [0,t]} \|\uGalerkin(s)\|^2_{\Hone}.
			\end{aligned}
		\end{equation}
	Concerning the boundary terms in \eqref{energy_ineq_1 Step 1 low_}, we integrate by parts in time:
		\begin{equation}
			\begin{aligned}
				\inttG (c^2 g+ b \ptalpha g) \utGalerkin \dGs= \intG (c^2 g+ b \ptalpha g) \uGalerkin \dG \Big \vert_0^t - \inttG (c^2 g+ b \ptalpha g)_t \uGalerkin \dGs.
			\end{aligned}
		\end{equation}
		Since $\uGalerkin\vert_{t=0}=0$ and $(\frakKonenegalpha*g_t)_t = \frakKonenegalpha*g_{tt}+ \frakKonenegalpha(t)g_t(0)$, we obtain
		\begin{equation}
			\begin{aligned}
				&\inttG (c^2 g+ b \ptalpha g) \utGalerkin \dGs\\
				= &\, \intG (c^2 g+ b \ptalpha g)(t) \uGalerkin(t) \dG - \inttG (c^2 g_t+ b \partial_t^{1+\alpha}g + b \frakKonenegalpha g_t(0)) \, \uGalerkin \dGs.
			\end{aligned}
		\end{equation}
		Using H\"older's inequality and the trace theorem, we estimate
		\begin{equation}
			\begin{aligned}
				\inttG (c^2 g+ b \ptalpha g) \utGalerkin \dGs
				\lesssim&\, \begin{multlined}[t] \| g(t)\|^2_{\HneghalfG}+ \| \ptalpha g(t)\|^2_{\HneghalfG}+\eps \|\uGalerkin(t)\|^2_{\Hone}\\
					+\|\frakKonenegalpha\|^2_{L^1(0,T)}\|g_t(0)\|^2_{\HneghalfG}+\eps \max_{s \in [0,t]} \|\uGalerkin(s)\|^2_{\Hone}\\
					+ \|g_t\|^2_{L^2(\HneghalfG)} + \|\partial_t^{1+\alpha} g\|^2_{L^2(\HneghalfG)} 
				\end{multlined}
			\end{aligned}
		\end{equation}
		for any $\eps>0$, where we have $\|\frakKonenegalpha\|_{L^1(0,T)} = \frac{T^{1-\alpha}}{\Gamma(2-\alpha)}\leq C$ for some $C>0$ independent of $\alpha$. By the assumed regularity of $g$, we have $g \in H^2(0,T; \HneghalfG) \hookrightarrow C^1([0,T]; \HneghalfG)$. 	
		Additionally, 
		\begin{equation} \label{rhs est low}
			\begin{aligned}
			 \intt	 ( F,  \utGalerkin)_{\Ltwo} \ds \lesssim&\, 
				 \|F \|^2_{L^2(\Ltwo)}+ \intt \| \utGalerkin\|^2_{\Ltwo}\ds.
			\end{aligned}
		\end{equation}
	Substituting \eqref{Est_1 low}, \eqref{Est_2 low}, and \eqref{rhs est low} into \eqref{energy_ineq_1 Step 1 low_} yields
		\begin{equation} \label{energy_ineq_1 Step 1_1 low}
			\begin{aligned}
				&\frac12\nLtwo{\sqrt{\fraka(t)}  \utGalerkin(t)}^2  + \csqhalf\nLtwo{\nabla \uGalerkin(t)}^2  + b
				\Calpha(T) \| \nabla \ptalpha  \uGalerkin \|^2_{\LtwotLtwo}  \\
				\lesssim&\,\begin{multlined}[t] 
					\int_0^t \left\{1+\left\|\frakl\right\|_{\Linf} +\left\|\frakn\right\|^2_{\Lthree}+\left\|\frakat\right\|_{\Linf}\right\}\nLtwo{ \utGalerkin}^2\ds\\+\eps \max_{s \in [0,t]} \|\uGalerkin(s)\|^2_{\Hone}
					+ \|F\|^2_{\LtwoLtwo}
					+ \inttG (c^2 g+ b \ptalpha g) \utGalerkin \dGs .
				\end{multlined}
			\end{aligned}
		\end{equation}
	 Adding the elementary estimate $\|\uGalerkin(t)\|^2_{\Ltwo} \lesssim T \|\utGalerkin\|^2_{\LtwotLtwo}+\|u_0^{(m)}\|^2_{\Ltwo}$ to \eqref{energy_ineq_1 Step 1_1 low},  results in the energy inequality
		\begin{equation} \label{step 1: energy_ineq  low}
			\begin{aligned}
				&\nLtwo{\sqrt{\fraka}  \utGalerkin(t)}^2  + \|\uGalerkin(t)\|^2_{\Hone}  + b
				\Calpha(T) \| \nabla \ptalpha  \uGalerkin \|^2_{\LtwotLtwo} \\
				\lesssimT&\,\begin{multlined}[t] 
					(1+\|\fraka\|_{W^{1, \infty}(\Linf)}+\|\frakl\|_{L^\infty(\Linf)}+\|\frakn\|^2_{L^2(\Lthree)})\| \utGalerkin\|_{\LtwotLtwo}^2\\ +\nLtwo{\uonem}^2 
					+ \|F\|^2_{\LtwoLtwo}
				+\|g\|^2_{H^2(\HneghalfG)}	+ \eps \max_{s \in [0,t]} \|\uGalerkin(s)\|^2_{\Hone}.
				\end{multlined}
			\end{aligned}
		\end{equation}
		for $t \in [0, T]$.  Thanks to the non-degeneracy assumption in \eqref{nondegeneracy assumption frakm frakm}, we have
		\[
		\ulfraka \nLtwo{\uGalerkin_t(t)}^2 \leq \nLtwo{\sqrt{\fraka(t)}  \uGalerkin_t(t)}^2.
		\] 
		 Employing this lower bound in  \eqref{step 1: energy_ineq  low} and taking the maximum over $t \in [0, \sigma]$ for $\sigma \in (0,T)$, we obtain, with sufficiently small $\eps$,  
					\begin{equation} \label{step 1: energy_ineq  low final}
	\begin{aligned}
		&  \|  \utGalerkin(\sigma)\|_{\Ltwo}^2  + \|\uGalerkin(\sigma) \|^2_{\Hone}  + b
		\ulC(T) \| \nabla \ptalpha  \uGalerkin \|^2_{L^2(0, \sigma;\Ltwo)}  \\
		\lesssim&\,\begin{multlined}[t] 
			\nLtwo{\uonem}^2 +
			(1+\|\fraka\|_{W^{1, \infty}(\Linf)}+\|\frakl\|_{L^\infty(\Linf)}+\|\frakn\|^2_{L^2(\Lthree)})\| \utGalerkin\|_{L^2(0,\sigma;\Ltwo)}^2\\
			+ \|F\|^2_{\LtwoLtwo}
			+\|g\|^2_{H^2(\HneghalfG)}.
		\end{multlined}
	\end{aligned}
\end{equation}
Applying \Gronwall's inequality and using the estimate $\|\uonem\|_{\Ltwo} \lesssim \|u_1\|_{\Ltwo} $, we obtain the semi-discrete version of the claimed estimate:
	\begin{equation} \label{final est lower m} 
	\begin{aligned}
		\|\uGalerkin\|^2_{\Xlow} \lesssim &\, \begin{multlined}[t] \Lambda_0(\fraka, \frakl, \frakn)	 \left(	\nLtwo{ u_1}^2 
			+ \|F\|^2_{\LtwoLtwo}
			+\|g\|^2_{H^2(\HneghalfG)} \right)\end{multlined}
	\end{aligned}
\end{equation}
where $\Lambda_0$ is defined in  \eqref{def Lambda_0}. \\[2mm]

		\noindent $\bullet$ \underline{Passing to the limit as $m \rightarrow \infty$}: 	Thanks to the $m$-uniform bound in \eqref{final est lower m}, we may extract subsequences, that we do not relabel, such that
								\begin{equation}
			\begin{aligned}
				\uGalerkin &\overset{\ast}{\rightharpoonup} u \quad  && \text{in} \quad L^\infty(0,T; \Hone), \qquad
				\utGalerkin \overset{\ast}{\rightharpoonup} \ut \quad  \text{in} \quad  L^\infty(0,T; \Ltwo), \\
				\ptalpha \uGalerkin &{\rightharpoonup} \, \zeta \quad  && \text{in} \quad L^2(0,T; \Hone).
			\end{aligned}
		\end{equation}
		We can identify $\zeta= \ptalpha u $ following the arguments in the proof of~\cite[Theorem 6.19, Sec.\ 6.4.2]{jin2021fractional}. 
		 Passing to the limit $m \rightarrow \infty$ in 
		\begin{equation} \label{weak linearized problem low m} 
			\begin{aligned}
			&	\begin{multlined}[t]-	\intTO (\fraka \varphi)_t \, \utGalerkin \dxt +\intTO (\csq \nabla \uGalerkin +b \nabla \ptalpha \uGalerkin)\cdot \nabla \varphi \dxt 
					\\+ \intTO  (\frakl  \utGalerkin+ \frakn \uGalerkin) \varphi \dxt - \intTG (c^2 g+ b \ptalpha g) \varphi \dGt   
				\end{multlined}\\
									=&\,   \intO  \fraka(0)\uonem \varphi(0) \dx+  \intTO F \varphi \dxt
			\end{aligned}
		\end{equation}
		and arguing that $u\vert_{t=0} =0$, we obtain a solution $u$ of \eqref{weak linearized problem low}. 
		\end{proof}
		The previous result ensures existence but does not yet imply uniqueness. Uniqueness follows under higher regularity of the variable coefficients $\fraka$ and $\frakl$.
		\begin{lemma} \label{lemma: uniqueness} 
			Let the assumption of Proposition~\ref{prop: wellp forward lower} hold.	If, in addition, 
			\[
			\fraka \in H^2(0,T; \Lfour) \ \text{ and } \ \frakl \in H^1(0,T; \Lthree),
			\]
			 then the solution of \eqref{weak linearized problem low} is unique.
		\end{lemma}
	\begin{proof}
To conclude uniqueness, we show that the homogeneous problem \eqref{weak linearized problem low} with $F=0$, $u_1=0$, and $g=0$ admits only the trivial solution $u=0$. We adapt the argument used in~\cite[Theorem 4, Ch.\ 7.2]{evans2010partial} and for a fixed $0\leq s\leq T$, define
		\begin{equation} \label{def Phi}
			\Phi(t)= \left\{
			\begin{aligned}
				&\int_t^s u(\tau)d \tau,&\quad 0\leq t\leq s, \\
				&0 ,&\quad s\leq t\leq T.
			\end{aligned}
			\right. 
		\end{equation}
	Clearly,
	\[
	\Phi \in L^2(0,T; \Hone) \cap H^1(0,T; \Ltwo), \quad \Phi(s) =0,
	\]
	and differentiation yields $\Phi_t(t) = - u(t)$ for $t \in (0,s)$.
	Using $\Phi$ as the test function in the homogeneous version of \eqref{weak linearized problem low} and  recalling that $u_t(0)=0$, we obtain
		\begin{equation}
			\begin{aligned}
	&	-\int_0^s 
			\int_\Omega \frakm(x,t) \ut\Phi_t\dxt
			+\csq \int_0^s\int_\Omega \nabla u \cdot \nabla\Phi\dxt\\
			 =&\, 	\begin{multlined}[t] \int_0^s\int_\Omega \frakmt u_t\Phi\dt-b \int_0^s\int_\Omega   \nabla \ptalpha u \cdot \nabla \Phi\dxt 
			-\int_0^s \intO (\frakl \ut + \frakn u)\Phi \dxt. 
			\end{multlined}
			\end{aligned}
		\end{equation}
		Substituting $\Phi_t=-u$ on $(0,s)$ yields
		\begin{equation}
			\begin{aligned}
				&\frac{1}{2}\int_0^s 
				\ddt \left( \|\sqrt{\frakm}  u\|_{\Ltwo}^2 
				-\csq \|\nabla \Phi\|_{\Ltwo}^2\right) \dt\\
				=&\,
				\begin{multlined}[t]\frac12 \int_0^s (\frakm_t u,  u)_{\Ltwo}\dt+ \int_0^s\int_\Omega \frakm_t \ut\Phi\dt- b\int_0^s\int_\Omega  \nabla \ptalpha u \cdot \nabla \Phi\dt
					-\int_0^s \intO (\frakl \ut + \frakn u)\Phi \dxt. 
				\end{multlined}
			\end{aligned}
		\end{equation}
		Using $\Phi(s)=0$ and the fact that $u(0)=0$ yields
		\begin{equation}\label{Estimate_1_Unq}
			\begin{aligned}
				&\frac{1}{2}(\|\sqrt{\frakm}  u(s)\|_{\Ltwo}^2+\csq \|\nabla \Phi(0)\|_{\Ltwo}^2)\\
				=&\,\begin{multlined}[t] \int_0^s\int_\Omega \frakm_t \ut\Phi\dt-\int_0^s\int_\Omega  b \nabla \ptalpha u \cdot \nabla \Phi\dt
				+\frac12 \int_0^s (\frakm_t u,  u)_{\Ltwo}\dt\\	-\int_0^s \intO (\frakl \ut + \frakn u)\Phi \dxt.
				\end{multlined}
			\end{aligned}
		\end{equation}
	Moreover, by definition of $\Phi$ and the Cauchy--Schwarz inequality
		\begin{equation}\label{Estimate_2_Unq}
			\begin{aligned}
				\|\Phi(0)\|_{\Ltwo}^2\leq \left(\int_0^s\| u(t)\|_{\Ltwo}\dt\right)^2\leq s \int_0^s \|u(t)\|_{\Ltwo}^2\dt
				\leq T\int_0^s \|u(t)\|_{\Ltwo}^2\dt.
			\end{aligned}
		\end{equation}
		Collecting \eqref{Estimate_1_Unq} and \eqref{Estimate_2_Unq}, and recalling the non-degeneracy assumption \eqref{nondegeneracy assumption frakm frakm}, we obtain 
		\begin{equation}\label{Estimate_3_Unq}
			\begin{aligned}
				&\frac{1}{2}(\| u(s)\|_{\Ltwo}^2+\csq \| \Phi(0)\|_{\Hone}^2)+b \int_0^s\int_\Omega  \nabla \ptalpha u \cdot \nabla \Phi\dt\\
				\lesssim &\, \int_0^s(T^2+\|\frakm_t\|_{\Linf})\|\sqrt{\frakm}  u(t)\|_{\Ltwo}^2\dt
				+\int_\Omega \frakm_t \ut \Phi\dt -\int_0^s \intO (\frakl \ut + \frakn u)\Phi \dxt.
			\end{aligned}
		\end{equation}
We treat next the second integral on the right-hand side.	Integrating by parts in time and using $\Phi(s)=0$, $u(0)=0$, and $\Phi_t = -u$ on $(0,s)$, we obtain 		
		\begin{equation}\label{Secon_Term}
			\begin{aligned}
				\int_0^s\int_\Omega \frakm_t \ut\Phi\dt = \intO \frakmt u \Phi \dx \ \Big \vert_0^s - \int_0^s \intO (\fraka_{tt} \Phi + \frakat(- u)) u \dxt.
			\end{aligned}
		\end{equation}
	Using $\Phi(s)=0=u(0)=0$, the first term on the right vanishes, and we arrive at
		\begin{equation}
			\begin{aligned}
				\int_0^s\int_\Omega \frakat \ut \Phi\dt 
				\lesssim&\, \int_0^s \|\frakm_{tt}\|_{\Lfour} \|\Phi\|_{\Lfour}\|u\|_{\Ltwo}\dt + \int_0^s \|\frakmt\|_{\Linf}\|u\|^2_{\Ltwo}\dt.
			\end{aligned}
		\end{equation}
Defining $U(t) = \int_0^t u(\tau)\, \textup{d}\tau$, and noting that $\Phi(t) = U(s) - U(t)$, we use the Sobolev embedding $H^1(\Omega) \hookrightarrow L^4(\Omega)$ to obtain
		\begin{equation}\label{Estimates_Second_term}
			\begin{aligned}
				&\int_0^s \|\frakm_{tt}(t)\|_{\Lfour} \|\Phi(t)\|_{\Lfour}\|u(t)\|_{\Ltwo}  \dt \\
				\lesssim &\, \int_0^s  \left(\|U(s)-U(t)\|_{\Hone}^2+\|\frakm_{tt}\|_{\Lfour}^2\|u\|_{\Ltwo}^2\right)\dt\\
				\lesssim &\, s\|U(s)\|_{\Hone}^2+\int_0^s  (\|U\|_{\Hone}^2+\|\frakm_{tt}\|_{\Lfour}^2\|u\|_{\Ltwo}^2)\dt.
			\end{aligned}
		\end{equation}
		Similarly, integrating by parts in time, we obtain
		\begin{equation} 
			\begin{aligned} 
					&-\int_0^s \intO (\frakl \ut + \frakn u)\Phi \dxt\\
					=&\,	-\int_0^s \intO  \frakn u \Phi \dxt +\int_0^s \intO    (\frakl_t \Phi -\frakl u ) u \dxt \\
						=&\,	\int_0^s \|\frakn\|_{\Lthree}\| u\|_{\Ltwo} \|\Phi\|_{\Lsix} \dt +\int_0^s    (\|\frakl_t\|_{\Lthree}\| \Phi\|_{\Lsix} +\|\frakl\|_{\Linf}\| u\|_{\Ltwo} )\| u \|_{\Ltwo} \dt \\
					\lesssim&\,	\begin{multlined}[t] \int_0^s  \left(\|\Phi\|^2_{\Hone} +  \left( \|\frakn\|^2_{\Lthree}  +\|\frakl_t\|_{\Lthree}^2+ \|\frakl\|_{\Linf}\right)\| u\|^2_{\Ltwo}\right) \dt.
									\end{multlined}
			\end{aligned}
		\end{equation}
Using the representation $\Phi(t) = U(s) - U(t)$,  we further obtain
				\begin{equation} \label{Estimates_Third_term}
			\begin{aligned} 
				&-\int_0^s \intO (\frakl \ut + \frakn u)\Phi \dxt\\
				\lesssim&\,	 \begin{multlined}[t] s\|U(s)\|_{\Hone}^2 + \int_0^s  \left\{\|U\|^2_{\Hone}  +\left( \|\frakn\|^2_{\Lthree}+\|\frakl_t\|_{\Lthree}^2+ \|\frakl\|_{\Linf}\right)\|u\|^2_{\Ltwo}\right\} \dt.
				\end{multlined}
			\end{aligned}
		\end{equation}
		Inserting the above  estimates  into \eqref{Estimate_3_Unq}, we obtain 
		\begin{equation} \label{est adjoint}
			\begin{aligned}
				&\frac{1}{2}(\| u(s)\|_{\Ltwo}^2+(\csq-C_0 s) \| U(s)\|_{\Hone}^2)+ b \int_0^s\int_\Omega   \nabla \ptalpha u \cdot \nabla \Phi\dt\\
				\lesssim &\, \begin{multlined}[t] \int_0^s(T^2+\|\fraka_t\|_{\Linf}+\|\frakl\|_{\Linf})\|u\|_{\Ltwo}^2 \dt\\
					+\int_0^s  (\|W\|_{\Hone}^2+(\|\frakm_{tt}\|_{\Lfour}^2+\|\frakn\|^2_{\Lthree}+\|\frakl_t\|^2_{\Lthree})\|u\|_{\Ltwo}^2)\dt,
				\end{multlined}
			\end{aligned}
		\end{equation}
		where $C_0$ is the sum of the hidden constants in \eqref{Estimates_Second_term} and \eqref{Estimates_Third_term}.  \\
		\indent It remains to show that the time-fractional term is non-negative. Since we will integrate by parts to prove this, we first regularize. Let $u_n \in {_0C}^1([0,s]; \Hone)$ and define
		$\Phi_n(t) := \int_{t}^{s} u_n(\tau)\textup{d}\tau$. Then
	\begin{equation}
	\begin{aligned}
		b	\int_0^s\int_\Omega   \nabla \ptalpha u_n \cdot \nabla \Phi_n \dt  =&\, 		b	\int_0^s\int_\Omega  \left( J^{-\alpha} \nabla  u_n \right) \cdot \nabla \Phi_n\dt \\
		=&\, 		b	\int_0^s \int_\Omega  (J^{-1} J^{1-\alpha }\nabla  u_n )  \cdot \nabla \Phi_n\dt.
	\end{aligned}
\end{equation}
Since $J^{-1} v= \partial_t v$ for $v \in {_0H^1(0,T)}$ (see~\cite[Sec.\ 2.3]{kubicaTimefractionalDifferentialEquations2020}), integrating by parts in time yields
		\begin{equation} \label{lower bound fract}
			\begin{aligned}
				b	\int_0^s\int_\Omega   \nabla \ptalpha u_n \cdot \nabla \Phi_n\dt 
				=&\, 		b	\int_0^s \int_\Omega   \left(J^{1-\alpha } \nabla  u_n \right) \cdot \nabla u_n \dt \ \geq 0 ,
			\end{aligned}
		\end{equation}
because 
$\Phi_n(s)=0$,  $J^{1-\alpha}\nabla u_n(0)=0$, and inequality \eqref{coercivity ineq 2} holds.  By density, \eqref{lower bound fract} extends to $u \in \Hsubalpha(0,s; \Hone)$.
		Thus from \eqref{est adjoint} we have 
		\begin{equation} \label{est adjoint2}
			\begin{aligned}
				&\frac{1}{2}\|\sqrt{\frakm}  u(s)\|_{\Ltwo}^2+(\csq-c_0 s) \| U(s)\|_{\Hone}^2 \\
				\lesssim &\, \begin{multlined}[t] \int_0^s(T^2+\|\fraka_t\|_{\Linf}+\|\frakl\|_{\Linf})\|\sqrt{\frakm}  u( t)\|_{\Ltwo}^2 \dt\\
					+\int_0^s  (\|U(t)\|_{\Hone}^2+(\|\frakm_{tt}\|_{\Lfour}^2+\|\frakn\|^2_{\Lthree}+\|\frakl_t\|^2_{\Lthree})\|\sqrt{\frakm}u\|_{\Ltwo}^2)\dt,
				\end{multlined}
			\end{aligned}
		\end{equation}
		from which we can conclude via \Gronwall's inequality that $u \equiv 0$ on $(0,s)$ for small enough $s$. We can apply the same arguments starting from time $s$ to extend the conclusion to $[0,T]$.
	\end{proof}

		\subsection{Analysis of a linearized wave problem:  Higher-regularity regime}		\label{sec: analysis higher order}				
				\indent To prove the well-posedness of the nonlinear state problem, we need to employ a higher-order testing procedure for the linearized problem that mimics the one used in the analysis of a linearized Westervelt equation with time-fractional attenuation in the setting of homogeneous Dirichlet data in~\cite{baker2024numerical}.  However, to be in a position to employ an analogous testing procedure to the Dirichlet case, we need to devise a suitable Neumann extension of the boundary data and at first work with smoother data $\greg$. We can then establish the required energy estimates on the regularized solutions, and afterward pass to the limit $\nu \rightarrow \infty$.  The following result will be used for that purpose. Throughout this and upcoming sections, we assume $\Omega \subset \R^d$, where $d \in \{1,2,3\}$ is a bounded  domain with sufficient regularity so that the following elliptic regularity property holds:
				\begin{equation} \label{assumption Omega}
					\text{	if $-\Delta z + z \in \Htwo$ \  and \ $\dfrac{\partial z}{\partial n} \in H^{5/2}(\partial \Omega)$, \ then $z \in H^4(\Omega)$.}
				\end{equation}
				We refer to, e.g.,~\cite[Theorem 2.5.1.1]{grisvard2011elliptic} for assumptions on $\Omega$ under which \eqref{assumption Omega} holds. 
\begin{lemma}[Compatible regularized boundary data] \label{lemma: Compatible reg data}					
	Let $\Gamma = \partial \Omega$, where $\Omega \subset \R^d$, $d \in \{1,2,3\}$, is a bounded domain that satisfies \eqref{assumption Omega}. Furthermore, let
	\begin{equation} \label{space g low}
	g \in  \spaceglower = \Hsubalpha(0,T; \HthreehalfG) \cap  {H}^1(0,T; \HonehalfG)   \cap H^2(0,T; \HneghalfG),
\end{equation}
	 and assume the compatibility conditions 
	\begin{equation}   \label{compatibility conditions}
			g(0)=0 \ \text{in} \ \ \HonehalfG, \qquad g_t(0)=0 \ \text{in} \ \ \HneghalfG. 
	\end{equation}
Then there exists a family $ \{\greg\}_{\nu} $ with 
	\[
	 \greg \in C^2([0,T]; \HfivehalfG), \quad  \greg(0)=\greg_t(0)=0,
	 \]
 such that, as $\nu \rightarrow \infty$,
\begin{equation} \label{g limits}
	\begin{aligned}
		 \greg &\rightarrow g &&\text{in } &&\Hsubalpha(0,T; \HthreehalfG), \qquad 	 \greg_t \rightarrow g_t \quad \text{in } \ L^2(0,T; \HonehalfG), \\
		 \greg_{tt} &\rightarrow g_{tt}  &&\text{in } &&L^2(0,T; \HneghalfG).
	\end{aligned}
\end{equation}
\end{lemma}
\begin{proof}
We conduct the proof by regularizing the function $g$ 
 and then employing a correction to guarantee homogeneous conditions at time zero, building upon the ideas in~\cite[Theorem 4]{milani1995compatible}.  \\[2mm]
\noindent $\bullet$ \underline{Regularization in time}:	We construct the family of mollifiers in time in the usual way; see~\cite[Ch.\ 1.4]{cherrier2012linear} and~\cite[Sec.\ 2.28]{adams2003sobolev}. Let $\eta \in C_0^\infty(\R)$ be the non-negative function supported in $\{|t| \leq 1\}$:
	\begin{equation}
		\eta(t) = \begin{cases}
			C_0 \exp{ \left(- \frac{1}{1-|t|^2}\right)} \quad &\text{if  }|t|<1, \\
			0 \quad &\text{if  } |t| \geq 1,
		\end{cases} 
	\end{equation}
	with $C_0$ chosen so that $\int_{-\infty}^\infty  \eta(t) \dt =1$. For $\eps>0$, we set
	\begin{equation}
		\etaeps(t) = \frac{1}{\eps} \eta\left(\frac{t}{\eps}\right).
	\end{equation}
	Let $\tildeg$ denote the extension of $g$ by zero to a function in $\R$ and introduce the time-mollification defined by
	\begin{equation}
		\tildeg^\eps(t) = (\etaeps * \gtilde)(t) =  \frac{1}{\eps}\int_{-\infty}^\infty \eta\left(\frac{t-\tau}{\eps}\right)\gtilde(\tau) \, \textup{d}\tau.
	\end{equation}
	Then $\tildeg^\eps \in C^\infty(\R; \HthreehalfG)$ and, setting $\geps= \tildeg^\eps_{\vert [0,T]}$, we have 
\begin{equation} \label{geps limits}
	\begin{aligned}
		\geps &\rightarrow g &&\text{  in  } &&\Hsubalpha(0,T; \HthreehalfG), \quad && \geps_t\rightarrow g_{t} &&\text{  in  } &&L^2(0,T; \HonehalfG), \\
			\geps_{tt} &\rightarrow g_{tt} &&\text{  in  } &&L^2(0,T; \HneghalfG), \quad && \geps_t(0) \rightarrow 0 &&\text{  in  } &&\HneghalfG,
	\end{aligned}
\end{equation}
as $\eps \searrow 0$.\\[2mm]
\noindent $\bullet$ \underline{Regularization in space}: By exploiting the fact that the space
$C^2([0,T]; \HfivehalfG)$ \sloppy  is dense in  $ C^2([0,T]; \HthreehalfG)$, we obtain a family $\{\gepstilde\}_{\nu \geq 1} \subset C^2([0,T]; \HfivehalfG)$ that converges to $g$ in the sense of \eqref{geps limits}. \\[2mm]
\noindent $\bullet$ \underline{Correction of regularized functions}:	We next introduce a correction to regularized functions in order to satisfy zero initial conditions:
	\begin{equation} \label{choice reg g}
	\greg(x,t)= \gepstilde(x,t)- \gepstilde(0)- t\rho(\Rnu t)   \gepstilde_t(0),
\end{equation}
where $\rho \in C_0^\infty(\R)$ is equal to $1$ in a neighborhood of zero and such that
\[
\int_0^\infty (s \rho(s)+ (1+s^2) \rho'(s) + s \rho''(s)) \ds \leq C;
\]
we include the construction of $\rho$ in Lemma~\ref{lemma: constructing r} in Appendix~\ref{appendix: density results} for completeness. Furthermore, we set
\begin{equation}
	\begin{aligned}
		\Rnu = \nu \left(1+ \|\gepstilde_t(0)\|^2_{\HthreehalfG}\right),
	\end{aligned}
\end{equation}
and 
$\eps= \frac{1}{\nu}$. The reason for including $1+ \|\gepstilde_t(0)\|^2_{\HthreehalfG}$ in the definition of $\Rnu$ is that the resulting terms involving $\gepstilde_t(0)$ below are then uniformly bounded; cf.~\eqref{conv reg first der}. From  \eqref{choice reg g}, we find that
\begin{equation} 
	\begin{aligned}
		\gregt=  \gepstilde_t - \chi_1(\Rnu t)\gepstilde_t(0),\quad
		\gregtt=&\,  \gepstilde_{tt}- \chi_2(\Rnu t)\gepstilde_t(0),	
	\end{aligned}
\end{equation}
where we have introduced the functions
\begin{equation} \label{def chi1 chi2}
	\begin{aligned}
		\chi_1(\Rnu t)=  \rho(\Rnu t)+  \Rnu t \rho'(\Rnu t), \quad
		\chi_2 (\Rnu t) = 2 \Rnu  \rho'(\Rnu t) + \Rnu^2 t \rho''(\Rnu t),
	\end{aligned}
\end{equation}
which correspond to the first and second derivatives of $\rho(\Rnu t)$, respectively. 
We see that the regularized functions satisfy compatibility conditions, that is, $\greg(0)  =0$ and  $\gregt(0) = 0$. \\
\indent  It remains to to establish the convergence of $\greg-\gepstilde$ and its derivatives to zero in suitable norms. Below $C_{1}, \ldots, C_4>0$ denote generic constants that do not depend on $\nu$. To establish the desired convergence, it helps to note that
\begin{equation} \label{conv chi1}
	\begin{aligned}
		\int_0^T |\ \chi_1(\Rnu t) |^2 \dt  \lesssim &\, 	\int_0^T ( t \rho(\Rnu t)+  \Rnu t^2 \rho'(\Rnu t))^2 \dt  \\
		\leq &\, \frac{1}{\Rnu^3} \int_0^\infty \left( s \rho(s)  +  s^2 \rho'(s)\right)^2 \ds \leq \frac{C_1}{\Rnu^3}. 
	\end{aligned}
\end{equation}
     We then find that
\begin{equation}  \label{conv reg first der}
	\begin{aligned}
		\| \gepstilde_t- \gregt\|_{L^2(\HonehalfG)} 
		\leq&\, 
		 \| \chi_1(\Rnu t)\|_{\LtwoT}\|\gepstilde_t(0) \|_{\HonehalfG} \\
		 \leq&\, C_2 \frac{1}{\nu^{3/2}}\frac{ \|\gepstilde_t(0)\|_{\HthreehalfG}}{\left(1+ \|\gepstilde_t(0)\|^2_{\HthreehalfG}\right)^{3/2}}
		 \leq C_3 \frac{1}{\nu^{3/2}}	 \rightarrow 0
	\end{aligned}
\end{equation}
as $\nu \rightarrow \infty$. Next, we have
\begin{equation}
	\begin{aligned}
		\|\ptalpha (\gepstilde- \greg)\|_{L^2(\HthreehalfG)} 
		= \begin{multlined}[t]  \|\ptalpha \left( \chi_1(\Rnu t)  \right) \gepstilde_t(0)\|_{L^2(\HthreehalfG)}. 
		\end{multlined}
	\end{aligned}
\end{equation}
Since we can write the fractional derivative term as the temporal convolution
\begin{equation}
	\ptalpha \left(t \rho(\Rnu t)\right)=  \frakKonenegalpha* \chi_1(\Rnu t), 
\end{equation}
an application of Young's convolution inequality yields
\begin{equation}
	\begin{aligned}
		\|\ptalpha \left( \chi_1(\Rnu t)\right) \gepstilde_t(0)\|_{L^2(\HthreehalfG)} 
		\leq&\, \|\frakKonenegalpha\|_{\LoneT}\| \chi (\Rnu t)\|_{\LtwoT} \|\gepstilde_t(0)\|_{\HthreehalfG} 
		\rightarrow\, 0
	\end{aligned}
\end{equation}
as $\nu \rightarrow \infty$, similarly to before. Therefore, we have $\|\ptalpha (\gepstilde-\greg)\|_{L^2(\HthreehalfG)}  \rightarrow 0$, which implies $\| \gepstilde- \greg\|_{\Hsubalpha(\HthreehalfG)}  \rightarrow 0$ as $\nu \rightarrow \infty$. \\
\indent Finally, we consider the convergence of $\gregtt$. Recalling \eqref{def chi1 chi2}, we have
\begin{equation}
	\begin{aligned}
		\int_0^T |\chi_2(\Rnu t)|^2 \dt  \lesssim \int_0^\infty (\rho'(s) + s \rho''(s))^2 \ds \leq C_4.
	\end{aligned}
\end{equation}
Since $\| \gepstilde_t(0)\|_{\HneghalfG} \rightarrow 0$ as $\nu=\frac{1}{\eps} \rightarrow \infty$,
we can conclude that 
\begin{equation}
	\begin{aligned}
		\int_0^T |\chi_2(\Rnu t)|^2 \dt   \cdot \|\gepstilde_t(0)\|^2_{\HneghalfG}  \rightarrow 0
	\end{aligned}
\end{equation}
as $\nu \rightarrow \infty$. This in turn allows us to infer that
\begin{equation} \label{conv reg second der}
	\begin{aligned}
		\|\gepstilde_{tt}- \greg_{tt}\|_{L^2(\HneghalfG)} 
		\leq \begin{multlined}[t]  \|\chi_2(\Rnu t)\|_{\LtwoT}  \|  \gepstilde_t(0)  \|_{\HneghalfG} \rightarrow 0,
		\end{multlined}
	\end{aligned}
\end{equation}
as $\nu \rightarrow \infty$, which completes the proof.
\end{proof}
Equipped with the previous density property, we now analyze the linearized wave problem in a higher-regularity setting. This additional smoothness will later allow us to employ the result to show the well-posedness of the nonlinear state problem.
\begin{proposition} \label{prop: wellp forward linear}
Let $T>0$, $c>0$,  $b \in (0,  \bar{b}]$ for some $\bar{b}>0$,
 and $\alpha \in (0,1)$. Assume that the variable coefficients in \eqref{linearized problem} satisfy
								\begin{equation}
									\begin{aligned}
										&\fraka \in \Xfrakm = L^\infty(0,T; \Wtwothree \cap \Woneinf) \cap W^{1,\infty}(0,T; \Linf) ,\\
										& \frakl \in \Xfrakl = L^\infty(0,T; \Htwo), \qquad
										\frakn \in \Xfrakn= L^\infty(0,T; \Htwo),  
									\end{aligned}
								\end{equation}
								and that the non-degeneracy condition \eqref{nondegeneracy assumption frakm frakm} holds. 
								Furthermore, assume that $u_1=0$, $F \in  L^2(0,T; \Htwo)$, and
	\begin{equation}  \label{def spaceg}
	\begin{aligned}
	g \in 	\spaceg =\begin{multlined}[t]  L^\infty(0,T; \HthreehalfG) \cap   \Hsubalpha(0,T; \HthreehalfG) \cap W^{1, \infty}(0,T; \HonehalfG)\\ \cap H^2(0,T; \HneghalfG)
		\end{multlined}
	\end{aligned}
\end{equation}
together with the compatibility conditions \eqref{compatibility conditions}.
Then there exists a unique solution $u $ of \eqref{linearized problem}, such that				
\begin{equation} \label{def_Xp}
	\begin{aligned}
	u \in	\Xp
		= \bigl\{ p \in L^{\infty}(0, T; \Hthree):&\  \pt \in  L^\infty(0, T;  \Htwo),\ \ptalpha p \in L^2(0,T; \Hthree)\\
		&\,								\ptt \in   L^2(0,T; \Hone) \bigr\}.
	\end{aligned}
\end{equation}
 Moreover, the solution satisfies the estimate
\begin{equation} \label{energy est forward}
									\begin{aligned}
										\|u\|^2_{\Xp} \lesssim &\, \begin{multlined}[t]		 \Lambda(\frakm, \frakn, \frakl)\left(
											\|F\|^2_{L^2(\Htwo)}+\|g\|^2_{\spaceg} \right), \end{multlined}
									\end{aligned}
								\end{equation}
								where 
								\begin{equation} \label{def Lambda}
									\begin{aligned}
										\Lambda(\fraka, \frakn, \frakl)=&\, \begin{multlined}[t]
											\calL^2_{\fraka, \frakn, \frakl}	\exp  \left\{ CT \calL^6_{\fraka, \frakn, \frakl} \right\}, \quad 	\calL_{\fraka, \frakn, \frakl} =  1+\|\fraka\|_{\Xfrakm}+\|\frakn\|_{\Xfrakn}+\|\frakl\|_{\Xfrakl},
										\end{multlined}
									\end{aligned}
								\end{equation}
								where the hidden constants are independent of $b$ and $\alpha$.
\end{proposition}
\begin{proof}
We conduct the proof through several steps. In the first step, we construct a Neumann extension of the boundary data, which allows us to reformulate the problem with homogeneous boundary and initial conditions. \\[2mm]								
								\noindent \underline{Step I: Neumann extension}. In the spirit of, e.g.,~\cite{kaltenbacher2011well, lasiecka1993uniform, bucci2019feedback}, where integer-order wave problems are analyzed, we introduce the following Neumann extension of $g$ 
								\begin{equation} \label{Neumann extension}
								\begin{aligned}
										-\Delta G+G = 0, \qquad
										  \frac{\partial G}{\partial n}=g
									\end{aligned}	
								\end{equation}
								pointwise in time. It is known by elliptic theory that $\Neumannextension: g \mapsto  G$ is a linear bounded mapping:
								\begin{equation} \label{extension}
									\Neumannextension: H^{s}(\Gamma)\rightarrow H^{s+\frac{3}{2}}(\Om),\qquad s\in \R
								\end{equation}
								at each $t \in [0, T]$;  see, e.g.,~\cite{Trigiani_1989, LASIECKA_1991}. Note that since $\Neumannextension$ is linear, it holds that $\partial_t \Neumannextension(\greg)=\Neumannextension(\partial_t \greg)$. By the compatibility conditions in \eqref{compatibility conditions}, 
								we conclude that $G(0)=\Gt(0)=0$. \\
								\indent We can then consider the following initial boundary-value problem for 
								\[
								w \coloneqq u- G
								\]
								 with homogeneous boundary and initial data:
								\begin{equation}\label{linearized problem homogen} \tag{IBVP$_\text{lin}^{\text{hom}}$}
									\left\{	\begin{aligned}
										&\ \frakm(x,t) \wtt - \csq \Delta w-b \Delta \ptalpha w+ \frakl(x,t)  \wt+ \frakn(x,t) w= \FG(x,t),\\
										&\ \frac{\partial w}{\partial n} =0 \quad \text{on } \ \Gamma, \\
										&\ (w, \wt) \vert_{t=0}= (0, 0),
									\end{aligned} \right.
								\end{equation}
								where the right-hand side is given by
								\begin{equation} \label{def F}
									\begin{aligned}
										\FG = F-\frakm(x,t) \Gtt + \csq \Delta G +b \Delta \ptalpha G - \frakl(x,t)  \Gt- \frakn(x,t) G.
									\end{aligned}
								\end{equation}
								~\\
								\noindent \underline{Step II: Regularized Neumann data}. Thanks to Lemma~\ref{lemma: Compatible reg data}, we can introduce the regularized Neumann data 
								\[
								\{ \greg\}_\nu \subset C^2([0,T];  \HthreehalfG), \quad \greg(0)=\gregt(0)=0,
								\]
							and instead of \eqref{linearized problem homogen} consider the following problem:
								\begin{equation}\label{linearized problem regularized}  \tag{IBVP$_\text{lin}^{\text{hom}, \nu}$}
									\left\{	\begin{aligned}
										&\ \frakm(x,t) \wregtt - \csq \Delta \wreg-b \Delta \ptalpha \wreg+ \frakl(x,t)  \wregt+ \frakn(x,t) \wreg= \FGreg,\\
										&\ \frac{\partial \wreg}{\partial n} =0 \quad \text{on } \ \Gamma, \\
										&\ (\wreg, \wregt) \vert_{t=0}= (0, 0).
									\end{aligned} \right.
								\end{equation}
								where the right-hand side is now
								\begin{equation} \label{def FGreg}
									\begin{aligned}
										\FGreg = F-\frakm(x,t) \Gregtt + \csq \Delta \Greg +b \Delta \ptalpha \Greg - \frakl(x,t)  \Gregt- \frakn(x,t) \Greg,
									\end{aligned}
								\end{equation}
								with $\Greg = \Neumannextension \greg$.
								By the regularity of $\greg$ and properties of $\Neumannextension$, we conclude that 
								\begin{equation}\label{Regularity_G}
								\Greg \in H^2(0,T; \Hfour).  
								\end{equation}
								Due to compatibility of regularized data, we also know that $\Greg(0)=0$ and $\Gregt(0)=0$. \\
								\indent The right-hand side of \eqref{linearized problem regularized}  can be bounded using the inequality
								\begin{equation}
									\begin{aligned}
									&\| v_1 v_2 \|_{L^2(\Hone)} \\
									\lesssim&\, \|v_1\|_{\LinfLinf}\|v_2\|_{\LtwoLtwo}+\|\nabla v_1\|_{L^\infty(\Lfour)}\|v_2\|_{L^2(\Lfour)}+\|v_1\|_{\LinfLinf}\|\nabla v_2\|_{\LtwoLtwo}\\
									\lesssim&\, \|v_1\|_{L^\infty(\Htwo)} \|v_2\|_{L^2(\Hone)}
									\end{aligned}
								\end{equation}
								as follows: 
								\begin{equation}\label{Estimate_F_nu}
									\begin{aligned}
										\|\FGreg\|_{\LtwoHone} 
										\lesssim&\, \begin{multlined}[t] \|F\|_{\LtwoHone}+ \|\frakm\|_{\Xfrakm}\|\Gregtt\|_{\LtwoHone}+\|\Delta \Greg\|_{\LtwoHone}\\+
										\bar{b}	\|\Delta \ptalpha \Greg\|_{\LtwoHone}+\|\frakl\|_{\Xfrakl}\|\Gregt\|_{\LtwoHone}+\|\frakn\|_{\Xfrakn}\|\Greg\|_{\LtwoHone}.
										\end{multlined}
									\end{aligned}
								\end{equation}
								 Then by  the properties of the extension mapping $\Neumannextension$, we have   
								\begin{equation} \label{est FGreg Hone}
									\begin{aligned}
										\|\FGreg\|_{\LtwoHone} 
										\lesssim \begin{multlined}[t] \|F\|_{\LtwoHone}+	\calLfraka \|\greg\|_{\spaceglower},
										\end{multlined}
									\end{aligned}
								\end{equation}
								where $\spaceglower$ is defined in \eqref{space g low}.\\[2mm]
								\noindent \underline{Step III: Galerkin approximation}. We next show the well-posedness of \eqref{linearized problem regularized} by employing the Galerkin procedure as in Proposition~\ref{prop: wellp forward lower}. We apply a testing procedure with test functions used in the Dirichlet case in~\cite[Lemma 3.2]{baker2024numerical}; we focus here on pointing out the main differences that arise due to the present Neumann setting. \\[1mm]	 						
								\noindent $\bullet$ {Testing with $\wregmt$}: We test the semi-discrete version of \eqref{linearized problem regularized} with $\wregmt$.
								Analogously to Proposition~\ref{prop: wellp forward lower}, this yields the energy estimate
								\begin{equation} \label{step 1: energy_ineq}
									\begin{aligned}
										&\nLtwo{\sqrt{\frakm(t)}\,  \wregmt(t)}^2  + \|\wregm(t)\|^2_{\Hone}  + b
										\ulC(T) \| \nabla \ptalpha  \wregm \|^2_{\LtwotLtwo} \\
										\lesssimT\;&\begin{multlined}[t] 
												\calLfraka^2 \| \wregmt\|_{\LtwotLtwo}^2
											+ \|\FGreg\|^2_{\LtwoLtwo}
										\end{multlined}
									\end{aligned}
								\end{equation}
								for $t \in [0, T]$. \\[2mm]								
								\noindent $\bullet$ {Testing with $-\Delta \wregmt$}:  Testing the semi-discrete equation with $-\Delta \wregmt$ and applying energy arguments analogous to~\cite{baker2024numerical} (see Appendix~\ref{appendix: energy est 1} for details) yields
								\begin{equation} \label{step 2: energy_ineq} 
									\begin{aligned}
										&\begin{multlined}[t] 	\|\sqrt{\fraka(t)}\nabla  \wregmt(t)\|^2_{\Ltwo} +	 	\|\D \wregm(t)\|^2_{\Ltwo} + b (\ulC(T) - \eps C_0)
										  \|\Delta \ptalpha \wregm \|^2_{\LtwotLtwo}
										\end{multlined}  \\
										\lesssim&\, 	\calLfraka^4\bigl(\|   \Delta \wregm\|^2_{\LtwotLtwo} +\|  \wregmt\|^2_{\LtwotHone}+\| \wregm\|^2_{\LtwotLtwo}  \bigr) \\&\hspace*{5.5cm}+	 \|\FGreg \|^2_{\LtwoLtwo}+\|\Delta \wregmt\|_{\LtwotLtwo}^2
									\end{aligned}
								\end{equation}
								for some $C_0>0$, independent of $b$ and $\alpha$. 
								We emphasize that we cannot yet control the term $\|\Delta \wregmt\|_{\LtwotLtwo}^2$ by the left-hand side. This becomes possible after the third testing step. \\[2mm]
							\noindent $\bullet$ {Applying $\Delta$ and testing with $\Delta \wregmt$}:  In the third testing step, we apply the operator $\Delta$ to the semi-discrete equation  associated to \eqref{linearized problem homogen} and then test the resulting equation with $\Delta \wregmt$. This yields
								\begin{equation} \label{step 3: identity_1}
									\begin{aligned} 
										&(\Delta(\frakm \wregmtt) -\csq \Delta^2 \wregm- b \Delta^2 \ptalpha \wregm,\,  \Delta \wregmt)_{\Ltwo}\\
										=&\, (\Delta \FGreg-\Delta(\frakl \wregmt)-\Delta(\frakn \wregm), \Delta \wregmt)_{\Ltwo}.
									\end{aligned}
								\end{equation}
							The terms involving variable coefficients $\fraka$, $\frakl$ and $\frakn$ are treated analogously to the Dirichlet case in \cite{baker2024numerical, kaltenbacher2022limiting}. The estimate of the $\FGreg$ term is new and we thus focus on it here.
								Recalling how $\FGreg$ is defined in \eqref{def FGreg} and using the fact that $\Delta \Greg=\Greg$ and $\Delta \ptalpha \Greg= \ptalpha \Greg$, we obtain
								\begin{equation}
									\begin{aligned}
										&\intt (\Delta \FGreg, \Delta \wregmt)_{\Ltwo} \ds \\
										=&\, \begin{multlined}[t]
											\intt \left(\Delta F -\Delta\left(\fraka\Gregtt \right) + \csq \Greg +b {\ptalpha  \Greg}  - \Delta \left(\frakl \Gregt\right)- \Delta \left(\frakn \Greg\right), \Delta \wregmt\right)_{\Ltwo} \ds.
										\end{multlined}
									\end{aligned}
								\end{equation}
For $\Delta(\fraka \Greg_{tt})$, we use the product rule and $\Delta \Greg_{tt}=\Greg_{tt}$:
\begin{equation}
	\int_0^t \bigl(\Delta(\fraka \Greg_{tt}), \Delta \wregmt\bigr)_{\Ltwo}\ds
	= \int_0^t \bigl( (\Delta \fraka)\,\Greg_{tt}
	+ 2\, \nabla \fraka \cdot \nabla \Greg_{tt}
	+ \fraka\, \Greg_{tt},\, \Delta \wregmt \bigr)_{\Ltwo}\ds .
\end{equation}
Analogously, we have
								\begin{equation}
									\begin{aligned}
									\intt \left(\Delta\left(\frakl \Gregt+\frakn \Greg\right), \Delta \wregmt\right)_{\Ltwo} 
										=&\,\begin{multlined}[t] \intt (\Delta \frakl \Gregt +2 \nabla \frakl \cdot \nabla \Gregt+\frakl \Gregt , \Delta \wregmt)_{\Ltwo}\\+\intt (\Delta \frakn \Greg +2 \nabla \frakn \cdot \nabla \Greg+\frakn \Greg , \Delta \wregmt)_{\Ltwo}.
											\end{multlined}
									\end{aligned}
								\end{equation}
							\noindent	By H\"older's inequality, we then obtain the estimate
								\begin{equation}
									\begin{aligned}
										&\intt (\Delta \FGreg, \Delta \wregmt)_{\Ltwo} \ds  \\
										\lesssim&\, \begin{multlined}[t]\intt \|\Delta F\|_{\Ltwo}\|\Delta \wregmt\|_{\Ltwo}  \ds +\calLfraka \intt \left(\|\Greg\|_{\Hone}+ \|\Gregt\|_{\Htwo}  +\|\ptalpha \Greg\|_{\Ltwo}\right. \\ \left.+\|\Gregtt\|_{\Hone} \right)  \|\Delta \wregmt\|_{\Ltwo}  \ds.
										\end{multlined}	
									\end{aligned}
								\end{equation}
								Above, we have relied on the bound
								\begin{equation}
									\begin{aligned}
										\int_{\Omega} \Delta \fraka \Gregtt \Delta \wregmt\dx\lesssim&\,\|\Delta \fraka\|_{\Lthree}\|\Gregtt\|_{\Lsix}\|\Delta \wregmt\|_{\Ltwo}\\
										\lesssim&\,\Vert \fraka \Vert_{\Wtwothree}\|\Gregtt\|_{\Hone}\|\Delta \wregmt\|_{\Ltwo},
									\end{aligned}
								\end{equation}
								and we have treated the $\frakl$ and $\frakn$ terms analogously. Then by elliptic regularity for the Neumann extension operator, we infer
										\begin{equation}
											\begin{aligned}
												& \intt (\Delta \FGreg, \Delta \wregmt)_{\Ltwo} \ds \\
												\lesssim&\, \begin{multlined}[t]\intt \left\{\|\Delta F\|_{\Ltwo}+\calLfraka \|\greg\|_{\spaceglower} \right\} \|\Delta \wregmt\|_{\LtwotLtwo}  \ds.
												\end{multlined}		
											\end{aligned}
										\end{equation}
										By testing the semi-discrete problem with $\wregmtt$ and its space-differentiated version with $\nabla \wregmtt$, it can be seen,  analogously to the Dirichlet case in~\cite{kaltenbacher2022limiting}, that 
										\begin{equation} \label{bound H1 wtt_}
											\begin{aligned}
												&\|  \wregmtt\|_{\LtwotHone} \\
												\lesssim&\,  \begin{multlined}[t] 
													\calLfraka^2\big ( \|  \Delta \ptalpha \wregm\|_{\LtwotHone} + \|   \wregm\|_{\LtwotHtwo}+ 
													\|\wregmt\|_{L^2(\Hone)}+\\ 
													\| \FGreg\|_{L^2(\Hone)} \bigr).
												\end{multlined}
											\end{aligned}
										\end{equation}
									 By otherwise also proceeding analogously to the Dirichlet case, we arrive at 
										 the outcome of the third testing step in the form of
										\begin{equation} \label{step 3: energy_ineq}
											\begin{aligned}
												&\begin{multlined}[t]  	\frac12\nLtwo{\sqrt{\fraka(t)} \Delta \wregmt(t)}^2  + \csqhalf\nLtwo{\nabla\D \wregm(t)}^2  \\+ b
													\ulC(T)  \|\nabla \D \ptalpha \wregm \|^2_{\LtwotLtwo}  
												\end{multlined} \\
												\lesssim&\,\begin{multlined}[t] 
													\calLfraka^6 \Bigl(\| \wregmt\|^2_{\LtwotHtwo}
													+\|\wregm\|_{\LtwotHtwo}^2		  \Bigr)
													+\eps b	\|  \Delta \ptalpha \wregm\|^2_{L^2_t(\Hone)} \\
													+\|\FGreg\|^2_{\LtwoHone}
													+ \|\Delta F\|_{\LtwoLtwo}^2 +\calLfraka^2\|\greg\|_{\spaceglower}^2
												\end{multlined}
											\end{aligned}
										\end{equation}
for any $\eps>0$ and all $t\in[0,T]$.  The skipped details are provided in Appendix~\ref{appendix: energy est 2} for completeness.
\\[2mm]	
\noindent $\bullet$ {Combining the estimates}:	Adding \eqref{step 1: energy_ineq}, \eqref{step 2: energy_ineq}, and \eqref{step 3: energy_ineq} and using the lower bound
										$\underline{\frakm} \le \frakm$ together with a suitable reduction of $\varepsilon>0$, we obtain										
										\begin{equation} \label{energy_ineq_2}
											\begin{aligned}
												&
												\|\wregmt(t)\|^2_{\Htwo}+ \|\wregm(t)\|^2_{\Hthree} 	+  \| \nabla  \ptalpha \wregm \|^2_{L^2_t(\Htwo)} \\
												\lesssim_T &\,\begin{multlined}[t] 	  \calLfraka^6 \Big( \| \wregmt\|^2_{\LtwotHtwo} 
													+ \|  \wregm\|^2_{L^2_t(\Htwo)}\Bigr) 
													+\|\FGreg\|^2_{\LtwoHone}+\|F\|^2_{\LtwoHtwo}
												\\	+\calLfraka^2 \|\greg\|_{\spaceglower}^2.
												\end{multlined}
											\end{aligned}
										\end{equation}									
Using the estimate for $\FGreg$ in~\eqref{est FGreg Hone} and combining with \eqref{bound H1 wtt_}, we further obtain
										\begin{equation} \label{energy_ineq_3}
											\begin{aligned}
												&
												\|{\wregmtt}\|^2_{\LtwotHone}+	\|\wregmt(t)\|^2_{\Htwo}+ \|\wregm(t)\|^2_{\Hthree} 	+  \| \nabla  \ptalpha \wregm \|^2_{L^2_t(\Htwo)}  \\
												\lesssim_T &\,\begin{multlined}[t] 	 \calLfraka^6  \Big( \| \wregmt\|^2_{\LtwotHtwo} 
													+ \|  \wregm\|^2_{L^2_t(\Htwo)}\Bigr) 
													+\| F\|^2_{\LtwoHtwo}+\calLfraka^2\|\greg\|_{\spaceglower}^2.
												\end{multlined}
											\end{aligned}
										\end{equation}
Applying Gr\"onwall's inequality to \eqref{energy_ineq_3} yields the uniform estimate
										\begin{equation} \label{final m est reg}
											\begin{aligned}
												&	\|\wregmtt\|^2_{\LtwoHone}+	\|\wregmt\|^2_{\LinfHtwo}+ \|\wregm\|^2_{\LinfHthree} 	+  \|\ptalpha \wregm \|^2_{\LtwoHthree}  \\
												\lesssimT&\, \Lambda(\frakm, \frakn, \frakl)\left( \|F\|^2_{\LtwoHtwo}+\|\greg\|^2_{\spaceglower}\right),
											\end{aligned}
										\end{equation}
where $\Lambda$ is defined in \eqref{def Lambda}.\\[2mm]
\noindent $\bullet$ {Passing to the limit as $m \rightarrow \infty$}: The estimates derived above imply that the sequence $\{\wregm\}_m$ is bounded uniformly in $m$ in the space $\Xp$. Hence, by standard weak compactness arguments, we can extract a (not relabeled) subsequence and pass to the limit in the semi-discrete formulation shows that $\wreg$ satisfies
\eqref{linearized problem regularized}. By taking the limit $m \to \infty$ in~\eqref{final m est reg} and using the lower
semi-continuity of norms, together with the fact that
\[
 \|\greg\|_{\spaceglower} \lesssim \|g\|_{\spaceglower},
\]
we obtain the $\nu$-uniform estimate
										\begin{equation} \label{final reg est}
											\begin{aligned}
												&	\|\wregtt\|^2_{\LtwoHone}+	\|\wregt\|^2_{\LinfHtwo}+ \|\wreg\|^2_{\LinfHthree} 	+  \|\ptalpha \wreg \|^2_{\LtwoHthree}  \\
												\lesssimT&\, \Lambda(\frakm, \frakn, \frakl)\left( \|F\|^2_{\LtwoHtwo}+\|g\|^2_{\spaceglower}\right).
											\end{aligned}
										\end{equation}

										\noindent \underline{Step IV: Passing to the limit as $\nu \rightarrow \infty$}. From the uniform estimate \eqref{final reg est}, $\{\wreg\}_{\nu}$ is bounded in $\Xp$. Hence, we may extract a subsequence (that we do  not relabel), such that
										\begin{equation}
											\begin{aligned}
												\wreg &\overset{\ast}{\rightharpoonup} w \quad   \text{in} \quad L^\infty(0,T; \Hthree), \quad
												&&\wregt \overset{\ast}{\rightharpoonup} \wt \quad   \text{in} \quad  L^\infty(0,T; \Htwo), \\
												\ptalpha \wreg &{\rightharpoonup} \ptalpha w \quad   \text{in} \quad L^2(0,T; \Hthree),\quad
												&&\wregtt {\rightharpoonup}  \wtt \quad   \text{in} \quad L^2(0,T; \Hone).
											\end{aligned}
										\end{equation}
										By the Aubin--Lions compactness theorem (see~\cite[(6.5)]{simon1986compact}), we also have strong convergence along a subsequence
										\begin{equation}
											\wreg \rightarrow w \quad \text{in} \quad C([0,T]; \Htwo) , \quad 	\wreg_t \rightarrow w_t \quad \text{in} \quad C([0,T]; \Hone). 
										\end{equation}
										Passing to the limit $\nu \rightarrow \infty$ in the regularized problem \eqref{linearized problem regularized} shows that $w$ solves \eqref{linearized problem homogen}. Uniqueness follows by testing the homogeneous problem with $\wt$, following an analogous estimating strategy as in the first testing step of the Galerkin procedure. \\
										\indent By passing to the $\nu \rightarrow \infty$ limit in \eqref{final reg est}, we see that $w$ satisfies
										\begin{equation} \label{final est w}
											\begin{aligned}
												&	\|\wtt\|^2_{\LtwoHone}+	\|\wt\|^2_{\LinfHtwo}+ \|w\|^2_{\LinfHthree} 	+  \|\ptalpha w \|^2_{\LtwoHthree}  \\
												\lesssimT&\, \Lambda(\frakm, \frakn, \frakl)\left( \|F\|^2_{\LtwoHtwo}+\|g\|^2_{\spaceglower}\right).
											\end{aligned}
										\end{equation}
										~\\
										\noindent \underline{Step V: The result for the inhomogeneous problem}. In the final step of the proof, we return to the function
										\begin{equation}
											\begin{aligned}
												u = w+ G
											\end{aligned}
										\end{equation} 
										which solves \eqref{linearized problem}. From estimate \eqref{final est w}, we have
										\begin{equation}
											\|w\|^2_{\Xp} \lesssim \Lambda(\frakm, \frakn, \frakl)\left( \|F\|^2_{\LtwoHtwo}+\| g\|^2_{\spaceglower}\right),
										\end{equation}
										and by the properties of the elliptic operator $\Neumannextension$, we have
										\begin{equation}
											\begin{aligned}
												\|G\|_{\Xp}=	\|\Neumannextension g\|_{\Xp} \lesssim \|g\|_{\spaceg}.
											\end{aligned}
										\end{equation}
Combining these bounds yields
\[
\|u\|^2_{\Xp}
\lesssim_T
\Lambda(\frakm,\frakn,\frakl)\bigl( \|F\|_{L^2(0,T;H^2)}^2 + \|g\|_{\spaceg}^2 \bigr),
\]
which is exactly \eqref{energy est forward}.  This completes the proof.
									\end{proof}
									\subsection{Well-posedness of the state problem}
Equipped with the well-posedness result for the linearized problem in the higher-regularity setting (Proposition~\ref{prop: wellp forward linear}), we now apply Banach’s fixed-point theorem to obtain existence and uniqueness for the full nonlinear state problem.
									\begin{theorem}\label{thm: wellp forward}
				Let $\alpha \in (0,1)$, $c>0$, $b \in (0, \bar{b})$ for some $\bar{b}>0$, and fix a final time $T>0$. Assume that the coefficient satisfies $k \in \Xk = W^{2,3}(\Omega) \cap W^{1,\infty}(\Omega)$, that the distributed source satisfies
				$f \in \Xf = L^2(0,T;H^2(\Omega))$, and that the boundary data $g \in \spaceg$ satisfies the
				compatibility conditions~\eqref{compatibility conditions}, where $\spaceg$ is defined in
				\eqref{def spaceg}. Then there exists a constant $\delta= \delta(T)>0$ such that if 
				\begin{equation}\label{smallness k}
					\|k\|_{\Xk} \leq \delta,
				\end{equation}
				the nonlinear problem \eqref{ibvp West} admits a unique solution $p \in \Xp$, where $\Xp$ is defined in \eqref{def_Xp}.
				The solution satisfies
				\begin{equation}\label{energy est forward nl}
					\|p\|_{\Xp}  \lesssim_T \|f\|_{\Xf} + \|g\|_{\spaceg},
				\end{equation}
				where the hidden constant is independent of the dissipation parameter $b$ and the fractional order $\alpha$. Moreover, there exist $\ulfrakm$, $\olfrakm$,  independent of $b$ and $\alpha$, such that
					\begin{equation}\label{nondegeneracy assumption frakm frakm}
						0<\ulfraka \leq1- 2k(x)p(x,t) \leq \olfraka \quad \ \text{for all $(x,t)\in \Omega  \times [0,T]$}.
				\end{equation}
									\end{theorem}
									~\\[-6mm]
										\begin{remark}[On the smallness assumption]
										Theorem~\ref{thm: wellp forward} requires smallness of the nonlinearity coefficient $k$. This coefficient can be expressed as 
										\[
										k = \frac{1+B/(2A)}{\rho c^2},
										\] 
										where $\rho$ denotes the background density and the ratio $B/A$ is the acoustic nonlinearity parameter, which is proportional to the ratio of the quadratic and linear coefficients in the Taylor expansion of pressure as a function of density perturbations; see~\cite{beyer2024parameter, prieur2011nonlinear}. In biological tissues, $B/A \in [5, 12]$; see \cite[Ch.\ 2, Table III]{hamilton1998nonlinear}. To assess the validity of the smallness assumption \eqref{smallness k}, we consider a nondimensional formulation obtained by introducing the scaling
										\begin{equation}
											\begin{aligned}
												x = L \tilde{x}, \qquad t = \frac{L}{c} \tilde{t}, \qquad p = \pref \, \tilde{p},
											\end{aligned}
										\end{equation}
										where $L$ is a characteristic length scale (e.g., the diameter of $\Omega$) and $\pref$ is a reference pressure amplitude.  After the change of variables and division by $\frac{c^2 \pref}{L^2}$, we obtain the dimensionless equation
										\begin{equation}
											\begin{aligned}
												\tilde{p}_{tt} - \Delta p - \tilde{b} \Delta \ptalpha \tilde{p} = \tilde{k} (\tilde{p}^2)_{tt},
											\end{aligned}
										\end{equation}
										where 
										\[
										\tilde{b} = \frac{b}{c^2 } \left(\frac{L}{c}\right)^{-\alpha} \quad \text{ and   } \quad \tilde{k} = \frac{(1+B/(2A)) \, \pref}{\rho c^2}.
										\]
										The parameter $\tilde{k}$ is  therefore small whenever $\pref \ll \rho c^2$. Typical ultrasound pressure amplitudes used in therapeutic acoustics are several orders of magnitude smaller than $\rho c^2 \sim 10^9$--$10^{10}\, \mathrm{Pa}$; see, for example,~\cite[Fig.\ 21.72]{beach2007medical}. Hence the smallness condition \eqref{smallness k} is meaningful in all practically relevant regimes.
									\end{remark}
										~\\[-7mm]
									\begin{proof}
										The proof of Theorem~\ref{thm: wellp forward} follows by applying Banach's fixed-point theorem to the mapping
										\begin{equation}
											\calT:  \pstar \mapsto p,
										\end{equation}
										where $\pstar$ is taken from a closed ball of radius $R$ in $\Xp$:
										\begin{equation}
											\begin{aligned}
												\ball = \left \{ \pstar  \in \Xp: \ \|\pstar\|_{\Xp} \leq R \right \},
											\end{aligned}
										\end{equation}
										and $p$ is defined as the solution of the linearized problem \eqref{linearized problem} with
										\begin{equation}
											\begin{aligned}  
												\fraka = 1- 2k\pstar, \quad \frakl = -2k \pstart, \quad \frak n =0, \quad \text{and} \quad F=f.
											\end{aligned}
										\end{equation}
										We first verify that the assumptions of Proposition~\ref{prop: wellp forward linear} are satisfied. 				The non-degeneracy condition on $\fraka$ in \eqref{nondegeneracy assumption frakm frakm} follows by employing the Sobolev embedding $\Htwo \hookrightarrow C(\overline{\Omega})$ (see~\cite[Theorem 4.12]{adams2003sobolev}):
										\begin{equation}
											\begin{aligned}
												\|2k \pstar\|_{\LinfLinf} \lesssim \|k\|_{\Linf} \|\pstar\|_{\LinfHtwo} \lesssim \|k\|_{\Linf} R.
											\end{aligned}
										\end{equation}	
										By choosing $\|k\|_{\Linf}$ sufficiently small, we ensure that $\fraka = 1-2k \pstar$ remains positive.
										The additionally required regularity 
										\[
										\fraka \in \Xfraka = L^\infty(0,T; \Wtwothree \cap \Woneinf) \cap W^{1,\infty}(0,T; \Linf)
										\]
										follows from the fact that $\Wtwothree$ and $\Woneinf$ are algebras in our at most three-dimensional setting $d \leq 3$ (see~\cite[Theorem 4.39]{adams2003sobolev}), so that
										\begin{equation}
											\begin{aligned}
												\|\fraka \|_{\Xfraka} \lesssim&\, 1+ \| k \pstar\|_{L^\infty( \Wtwothree \cap \Woneinf)} + \|k\|_{\Linf} \|\pstar\|_{W^{1,\infty}(\Linf)} \\
											 \lesssim&\, 1+ \| k\|_{\Xk} \| \pstar\|_{L^\infty( \Wtwothree \cap \Woneinf)} + \|k\|_{\Linf} \|\pstar\|_{W^{1,\infty}(\Linf)} \\	
												\lesssim&\, 1+ \|k\|_{\Xk}R,
											\end{aligned}
										\end{equation}
										where in the last step we have used the Sobolev embedding $\Hthree \hookrightarrow \Wtwothree \cap \Woneinf$. Similarly, 
										\begin{equation} \label{est frakn} 
											\|\frakl\|_{\LinfHtwo} \lesssim \|k\|_{\Htwo} R.
										\end{equation}
						 Therefore the assumptions of Proposition~\ref{prop: wellp forward linear} are fulfilled, and		$\mathcal{T}$ is well-defined. \\[1mm]
					\indent To show that $\calT(\ball) \subset \ball$, let $\pstar \in \ball$ and $p = \calT(\pstar)$. From Proposition~\ref{prop: wellp forward linear}, we have
										\begin{equation}
											\begin{aligned}
												\|p\|^2_{\Xp} \lesssim_T &\, \begin{multlined}[t]		 \Lambda(\fraka, \frakn, \frakl)\left(
													\|g\|^2_{\spaceg}+\|f\|^2_{\spacef} \right), \end{multlined}
											\end{aligned}
										\end{equation}
										where
										\begin{equation}
											\calLfraka = 1+ \|\fraka\|_{\Xfrakm} + \|\frakl\|_{\Xfrakl} \lesssim 1+ \|k\|_{\Xk} R,
										\end{equation}
										and 
										\begin{equation}
											\Lambda(\fraka, \frakn, \frakl) =	\calL^2_{\fraka, \frakn, \frakl}	\exp  \left\{ CT \calL^6_{\fraka, \frakn, \frakl} \right\} \lesssim (1+ \|k\|_{\Xk} R)^2 	\exp  \left\{ CT (1+ \|k\|_{\Xk} R)^6\right\}. 
										\end{equation}
										Hence,  
										\begin{equation}
											\begin{aligned}
												\|p\|^2_{\Xp} \lesssim &\, (1+ \|k\|_{\Xk} R)^2 	\exp  \left\{ CT (1+ \|k\|_{\Xk} R)^6\right\}(\gconst^2+\fconst^2) \leq R^2,
											\end{aligned}
										\end{equation}
										where the last inequality holds if the radius $R= R(\gconst, \fconst)$ is sufficiently large and $\|k\|_{\Xk}$ sufficiently small relative to $R$.  Thus, in that setting, $\calT(\ball) \subset \ball$. \\[1mm]
										\indent The verify that $\calT$ is a strict contraction, take $\pstarone$, $\pstartwo \in \ball$ and set $\pone= \calT \left( \pstarone\right)$, $\ptwo = \calT\left( \pstartwo\right)$. Let $\pstardiff = \pstarone-\pstartwo$. The difference $\pdiff=\pone-\ptwo$ solves
										\begin{equation}
											\begin{aligned}
												((1-2k \pstarone)\pdiff_t)_t - c^2 \Delta \pdiff - b \Delta \ptalpha \pdiff - 2k (\pstardiff \ptwo_t)_t =0
											\end{aligned}
										\end{equation}
										with homogeneous initial and Neumann data. Using the uniform bounds \[
										\|\pstarone\|_{\Xp}, \ \|\pstartwo\|_{\Xp} \lesssim R,
										\] 
										and exploiting the estimate \eqref{final est lower} from Proposition~\ref{prop: wellp forward lower} with 
										\[
										\fraka = 1-2k \pstarone, \ \frakl = 1-2k \pstarone_t, \ \frakn =0, \text{ and }\ F= - 2k (\pstardiff \ptwo_t)_t,
										\]
										we obtain strict contractivity of $\calT$ in the norm of \sloppy $W^{1,\infty}(0,T; \Ltwo) \cap L^\infty(0,T; \Hone)$  for $\|k\|_{\Linf}$ sufficiently small. We can argue that $\ball$ is closed in this norm by exploiting the Banach--Alaoglu theorem (see \cite[Theorem B.1.7]{hytonen2016analysis}), analogously to~\cite[Theorem 4.1]{kaltenbacher2022parabolic}. 									
										 Banach's fixed-point theorem now yields the existence of a unique fixed-point $p \in \ball \subset \Xp$, which is the unique solution of the problem.										
									\end{proof}
									  
\section{Existence of optimal controls} \label{sec: existence opt control}
In this section, we establish the existence of a minimizer for the optimal control under study. To this end, we first introduce the control-to-state operator. Let $\gconst$, $\fconst>0$ be given and fixed. We define the sets of admissible boundary controls as follows:
\begin{equation} \label{def spaceg admissible}
	\begin{aligned}
		\spacegad =&\,\begin{multlined}[t] \Bigl\{ g \in \spaceg=L^\infty(0,T; \HthreehalfG) \cap   \Hsubalpha(0,T; \HthreehalfG) \cap W^{1, \infty}(0,T; \HonehalfG)\\ \cap H^2(0,T; \HneghalfG)
			:  \|g\|_{\spaceg} \leq \gconst, \ g(\cdot, 0)= \gt(\cdot, 0)=0  \Bigr\},
		\end{multlined}
	\end{aligned}
\end{equation}
The set $\spacegad$ represents a set of smooth controls in the Banach space $\spaceg$ subject to appropriate compatibility conditions with respect to the initial data, as dictated by Theorem~\ref{thm: wellp forward}. The admissible set of distributed controls is defined as the closed ball in the Hilbert space $\spacef$
\begin{equation} \label{def spacef admissible}
	\begin{aligned}
		\Xfad =&\,\begin{multlined}[t] \Bigl\{ f \in \Xf = L^2(0,T; \Htwo): \ \|f\|_{\Xf} \leq \fconst \Bigr\}.
		\end{multlined}
	\end{aligned}
\end{equation}
The control-to-state mapping is then defined by
\begin{equation}\label{def controltostate}
	\begin{aligned}
		\controltostate:  \spacegad \times \Xfad \rightarrow \Xp, \qquad (g,f) \mapsto \, \controltostate(g,f)= p \ \ \text{solving}\ \eqref{ibvp West}.
	\end{aligned}
\end{equation} 
Under the assumptions of Theorem~\ref{thm: wellp forward} with the nonlinearity coefficient $\|k\|_{\Xk} \leq \delta$, the control-to-state mapping is well-defined. 
We can then write the optimal control problem \eqref{opt problem}
in the reduced form 
as
\begin{equation} \label{reduced problem}
	\begin{aligned}
\fbox{	$	\text{Find } (\tildeg, \tildef) \in \spacegad\times \Xfad, \text{ such that } \ j(\tildeg, \tildef) = \displaystyle \inf_{(g,f) \in \spacegad \times \Xfad} j(g, f)$}
	\end{aligned}
\end{equation}
with
\begin{equation} \label{def j reduced}
	j(g,f) = J(\controltostate(g,f), g, f). 
\end{equation}
We next address the question of existence of optimal controls.
\begin{theorem}[Existence of optimal controls] \label{thm: existence opt control}   
	Let the assumptions of Theorem~\ref{thm: wellp forward} hold and let $\controltostate$ be defined in \eqref{def controltostate}.
	Then there exists at least one optimal control $(\tildeg, \tildef) \in \spacegad \times \Xfad$,
	 which minimizes the cost functional $J(\controltostate(g,f), g, f)$ over $(g, f) \in \spacegad \times \Xfad$.
\end{theorem}
\begin{proof}
The proof follows via the direct method of calculus of variations; see, for example,~\cite[Ch.\ 9.2]{manzoni2021optimal}. We first note that $ \spacegad \times \spacefad \neq \emptyset$. Since $J$ is non-negative, the infimum 
\[
\tilde{j} = \displaystyle \inf_{(g,f)\in \spacegad \times \Xfad} J(\controltostate(g,f), g, f) 
\]
 exists.  Let $\left\{(\gm, \fm)\right\}_{m\geq 1} \subset \spacegad \times \spacefad$
	be a minimizing sequence, such that 
	\begin{equation}
	\tildej = 	\lim_{m \rightarrow \infty} J(\controltostate(\gm, \fm), \gm, \fm).
	\end{equation}
Then $\pressurem=\controltostate(\gm, \fm)$ is the solution of
	\begin{equation} \label{ibvp West m}
		\left\{    
		\begin{aligned}
			&((1-2k \pressurem)\pressurem_t)_t-\csq \Delta \pressurem - b \Delta \ptalpha \pressurem=\fm  \quad &&\text{in} \ \Omega \times (0,T),  \\
			& \frac{\partial {\pressurem}}{\partial n}=\gm \quad &&\text{on} \ \Gn \times (0,T), \\
			&(\pressurem, \pressurem_t)= (0,0) &&\text{on} \ \Gamma \times \{0\},
		\end{aligned}
		\right.   
	\end{equation}
	in the sense of Theorem~\ref{thm: wellp forward}.  On account of $\gm \in \spacegad$, $\fm \in \spacefad$, and Theorem~\ref{thm: wellp forward}, we have the uniform bounds
	\begin{equation}
\|\gm\|_{\spaceg}\leq L_1, \quad  \|\fm\|_{\spacef} \leq L_2, \quad	 \|\pressurem\|_{\Xp} \lesssim_T \|\gm\|_{\spaceg}+ \|\fm\|_{\spacef} \lesssim_T \gconst+\fconst.
	\end{equation}
Therefore, using the Banach--Alaoglu theorem, we may extract weakly(-star) convergent subsequences, that we do not relabel, such that
	\begin{equation} \label{weak conv pressurem}
		\begin{aligned} 
			\pressurem &\overset{\ast}{\rightharpoonup} \pstar \quad  && \text{in} \quad L^\infty(0,T; \Hthree), \quad
			&&\pressurem_t \overset{\ast}{\rightharpoonup} \pstar_t \quad  && \text{in} \quad  L^\infty(0,T; \Htwo), \\
			\ptalpha \pressurem &{\rightharpoonup} \ptalpha \pstar \quad  && \text{in} \quad L^2(0,T; \Hthree),\quad
			&&\pressurem_{tt} {\rightharpoonup}  \pstar_{tt} \quad &&  \text{in} \quad L^2(0,T; \Hone),
		\end{aligned}
	\end{equation}
	and
	\begin{equation} \label{weak conv gm}
		\begin{aligned}
			\gm &\overset{\ast}{\rightharpoonup} \tildeg \quad &&\text{in} \quad L^\infty(0,T; \HthreehalfG), \quad
			&&\gm_t  \overset{\ast}{\rightharpoonup} \tildeg_t && \text{in} \quad L^\infty(0,T; \HonehalfG), \\
			\ptalpha \gm &{\rightharpoonup} \, \ptalpha \tildeg&& \text{in} \quad L^2(0,T; \HthreehalfG), \quad
			&&\gm_{tt} {\rightharpoonup}\, \tildeg_{tt} && \text{in} \quad  L^2(0,T; \HneghalfG),
		\end{aligned}
	\end{equation}
	as well as
	\begin{equation}
		\begin{aligned}
			\fm & \rightharpoonup  \tildef &&\text{in} \quad L^2(0,T; \Htwo)
		\end{aligned}
	\end{equation}
	as $m \rightarrow \infty$. Further,  \eqref{weak conv pressurem} and the Aubin--Lions compactness theorem implies the following strong convergence:
	\begin{equation} \label{strong convergence pressurem}
		\begin{aligned}
			\pressurem &\rightarrow \pstar \  && \text{in} \ \, C([0,T]; \Htwo), \quad
			\pressurem_t \rightarrow \pstar_t \   \text{in} \ \,  C([0,T]; \Hone).
		\end{aligned}
	\end{equation}
	The weak lower semi-continuity of norms guarantees that 
	\[
	\|\tildeg\| \leq L_1,  \quad  \|\tildef\| \leq L_2.
	\]
	Further, \eqref{weak conv gm} and and \cite[Lemma 3.1.7]{zheng2004nonlinear} imply
	\begin{equation}
		\begin{aligned}
			\gm(0) &\overset{\ast}{\rightharpoonup}  \tildeg(0) \  && \text{in} \ \, \HonehalfG, \quad
			\gm_t(0) \overset{\ast}{\rightharpoonup}  \tildeg_t(0) \   \text{in} \ \, \HneghalfG.
		\end{aligned}
	\end{equation}
	By the uniqueness of limits, we have $\tildeg(0)=\tildeg_t(0)=0$. Thus, we conclude that $(\tildeg,\tildef) \in \spacegad \times \spacefad$. \\
	\indent	Now, to show that $\pstar = \controltostate(\tildeg, \tildef)$, we will pass to the limit $m \rightarrow \infty$ in the weak form of \eqref{ibvp West m}.  Since $\pressurem \rightarrow \pstar$ in $C([0,T];  \Htwo)$ strongly and $	\pressurem_{tt} \rightharpoonup  \pstar_{tt}$ in $L^2(0,T; \Hone)$ weakly, we have
	\begin{equation} \label{product limits}
 \pressurem \pressurem_{tt} \rightharpoonup \pstar \pstar_{tt} \qquad \text{in } \ L^2(0,T; \Hone).
	\end{equation}
The remaining terms can be handled analogously or in a simpler manner as $m \rightarrow \infty$, to conclude that $\pstar$ solves \eqref{ibvp West m}. 
\\ 
\indent Since $\pressurem \rightarrow \pstar$ in $C([0,T]; \Hone)$ strongly, we also have
	\begin{equation}
		\pressurem(T) \rightarrow \pstar(T) \quad \text{in} \ \Ltwo
	\end{equation}
	as $m \rightarrow \infty$. For the statement to follow, we then need the   lower semi-continuity of $J$:
	\begin{equation} \label{lsc}
		\begin{aligned}
			J(\pstar, \tildeg, \tildef) 
			\leq&\,	\liminf_{m \rightarrow \infty} 	J(\pressurem, \gm, \fm) \leq \tilde{j},
		\end{aligned}
	\end{equation}
	which holds due to the weak lower semi-continuity property of norms. By the definition of infimum, we conclude that
	\begin{equation}
			J(\pstar, \tildeg, \tildef) = \tilde{j},
	\end{equation}
which shows optimality.	
\end{proof}
Because the state problem is nonlinear, and the reduced cost functional thus non-convex, we cannot expect the uniqueness of minimizers without imposing additional assumptions; we refer to the discussion in~\cite[Sec.\ 4.4]{troltzsch2024optimal} for further details on this topic.
\subsection{Stability with respect to perturbations}
We next discuss the stability of minimizers with respect to perturbations in the targeted pressure distribution. Subsequently, motivated also by numerical simulations, we consider a setting with approximate optimal controls $(\ggamma, f_\gamma)$, such that
	\begin{equation} 
	J_\gamma^{\pdgamma}(\controltostate(\ggamma, \fgamma), \ggamma, \fgamma) \leq 	J_\gamma^{\pdgamma}(\controltostate(g, f), g, f)+o(\gamma)
	\end{equation}  
and  discuss the limiting behavior as $\gamma \searrow 0$.
\begin{proposition}[Stability with respect to perturbations in $\pd$]
	Let the assumptions of Theorem~\ref{thm: existence opt control} hold. Let $\eta>0$ in the regularizing functional \eqref{def calR}. If $\pressuremd \rightarrow \pd$ in $C([0,T]; \Ltwo)$ as $m \rightarrow \infty$, 
	then the corresponding minimizers $\{(\gm, \fm)\}_{m \geq 1}$ of 
	\begin{equation}
		\begin{aligned} 
			&J^{\pressuremd}(\controltostate(g,f), g, f)\\
			 =&\, \begin{multlined}[t] \frac{\nu}{2}\|\controltostate(g,f)-\pressuremd\|_{L^2(L^2(\Omega_0))}^2 +\frac{1-\nu}{2}\|\controltostate(g,f)(T)-\pressuremd(T)\|^2_{L^2(\Omega_0)} +\calR(g,f)
			\end{multlined}
		\end{aligned}
	\end{equation}
	over $\spacegad \times \spacefad$ have a strongly convergent subsequence in $C([0,T]; \HonehalfG) \times L^2(0,T; \Ltwo)$ and the limit of each convergent subsequence as $m \rightarrow \infty$ is a minimizer of $	J^{\pd}(\controltostate(g,f), g, f)$.
\end{proposition}
\begin{proof} 
	The proof follows by adapting ideas from~\cite[Theorem 2.1]{engl1989convergence}. Since $\{(\gm, \fm)\}_{m \geq 1}$ are minimizers, we have
	\begin{equation} \label{ineq}
		\begin{aligned}
			J^{\pressuremd}(\controltostate(\gm, \fm), \gm, \fm) \leq 	J^{\pressuremd}(\controltostate(g,f), g, f)
		\end{aligned}
	\end{equation}
	for any $(g, f) \in \spacegad \times \spacefad$. Let $\pressurem = \controltostate(\gm, \fm)$. Thanks to the uniform boundedness of $\{(\gm, \fm)\}_{m \geq 1}$ and $ \{\pressurem\}_{m \geq 1}$ (the latter, similarly to before, follows by Theorem~\ref{thm: wellp forward}), there exist subsequences, not relabeled, such that
	\begin{equation} \label{weak conv pressurem 1}
		\begin{aligned} 
			\pressurem &\overset{\ast}{\rightharpoonup} \pstar \quad  && \text{in} \quad L^\infty(0,T; \Hthree), \qquad
			&&\pressurem_t \overset{\ast}{\rightharpoonup} \pstar_t \quad  &&\text{in} \quad  L^\infty(0,T; \Htwo), \\
			\ptalpha \pressurem &{\rightharpoonup} \ptalpha \pstar \quad  && \text{in} \quad L^2(0,T; \Hthree),\qquad
			&&\pressurem_{tt} {\rightharpoonup}  \pstar_{tt} \quad   &&\text{in} \quad L^2(0,T; \Hone).
		\end{aligned}
	\end{equation}
	and
	\begin{equation} \label{weak conv gm 1}
		\begin{aligned}
			\gm &\overset{\ast}{\rightharpoonup} \tildeg \quad &&\text{in} \quad L^\infty(0,T; \HthreehalfG), \qquad
			&&\gm_t  \overset{\ast}{\rightharpoonup} \tildeg_t  &&\text{in} \quad L^\infty(0,T; \HonehalfG), \\
			\ptalpha \gm &{\rightharpoonup} \, \ptalpha \tildeg&& \text{in} \quad L^2(0,T; \HthreehalfG) \qquad
			&&\gm_{tt} {\rightharpoonup}\, \tildeg_{tt}  &&\text{in} \quad  L^2(0,T; \HneghalfG),
		\end{aligned}
	\end{equation}
	as well as
$\fm  \rightharpoonup  \tildef$ in $L^2(0,T; \Htwo)$
	as $m \rightarrow \infty$. 	By the Aubin--Lions theorem, we have the strong convergence along a subsequence
	\begin{equation}
		\gm \rightarrow \tildeg \ \text{ in } \ C([0,T]; \HonehalfG). 
	\end{equation}

	 We can then pass to the limit in the weak form of the problem solved by $\pressurem$ as $m \rightarrow \infty$ to obtain $\pstar = \controltostate(\tildeg, \tildef)$, analogously to the proof of Theorem~\ref{thm: existence opt control}. \\
\indent	By the weak lower semi-continuity of norms and \eqref{ineq},
	\begin{equation} \label{ineqq}
		\begin{aligned}
			J^{\pd}(\controltostate(\tildeg, \tildef), \tildeg, \tildef) \leq&\, \liminf_{m \rightarrow \infty} 	J^{\pressuremd}(\controltostate(\gm, \fm), \gm, \fm) \\
			\leq&\, \limsup_{m \rightarrow \infty} 	J^{\pressuremd}(\controltostate(\gm, \fm), \gm, \fm) \\
			\leq&\, \lim_{m \rightarrow \infty} 	J^{\pressuremd}(\controltostate(g, f), g, f) 
			=	J^{\pd}(\controltostate(g, f), g, f),
		\end{aligned}
	\end{equation}
	for all $(g,f) \in \spacegad \times \spacefad$, where we have used the strong convergence of $\pressuremd$ in the last step. This inequality implies that $(\tildeg, \tildef)$ is a minimizer and, combined with \eqref{ineq}, that
	\begin{equation} \label{id}
		\begin{aligned}
			\lim_{m \rightarrow \infty} 	J^{\pressuremd}(\controltostate(\gm, \fm), \gm, \fm) = 	J^{\pd}(\controltostate(\tildeg, \tildef), \tildeg, \tildef).
		\end{aligned}
	\end{equation}
	Indeed, from \eqref{ineq} we have, because $(\tildeg, \tildef) \in \spacegad \times \spacefad$, 	\begin{equation} 
		\begin{aligned}
			J^{\pressuremd}(\controltostate(\gm, \fm), \gm, \fm) \leq 	J^{\pressuremd}(\controltostate(\tildeg, \tildef), \tildeg, \tildef).
		\end{aligned}
	\end{equation}
	From here,
		\begin{equation} 
		\begin{aligned}
		\limsup_{m \rightarrow \infty} 	J^{\pressuremd}(\controltostate(\gm, \fm), \gm, \fm) \leq&\, \limsup_{m \rightarrow \infty} 	J^{\pressuremd}(\controltostate(\tildeg, \tildef), \tildeg, \tildef)\\
		=&\, \lim_{m \rightarrow \infty} 	J^{\pressuremd}(\controltostate(\tildeg, \tildef), \tildeg, \tildef)
				=	J^{\pd}(\controltostate(\tildeg, \tildef), \tildeg, \tildef).
		\end{aligned}
	\end{equation}
	On the other hand, due to \eqref{ineqq},
		\begin{equation} 
		\begin{aligned}
			J^{\pd}(\controltostate(\tildeg, \tildef), \tildeg, \tildef) \leq \liminf_{m \rightarrow \infty} 	J^{\pressuremd}(\controltostate(\gm, \fm), \gm, \fm). 
		\end{aligned}
	\end{equation}
	These two inequalities combined yield
	\begin{equation}
		\limsup_{m \rightarrow \infty} 	J^{\pressuremd}(\controltostate(\gm, \fm), \gm, \fm) \leq 	J^{\pd}(\controltostate(\tildeg, \tildef), \tildeg, \tildef) \leq	\liminf_{m \rightarrow \infty} 	J^{\pressuremd}(\controltostate(\gm, \fm), \gm, \fm),
	\end{equation}
	which allows us to conclude that \eqref{id} holds.\\	
\indent 	In the last step, we prove strong convergence of $\fm$ in $\LtwoTLtwo$. Assume that the opposite is true, that is, $\fm \nrightarrow \tildef$ in $L^2(0,T; \Ltwo)$. Since $\|\fm\|_{L^2(\Ltwo)}$ is bounded,  then
	\begin{equation}
		C \coloneqq \limsup_{m \rightarrow \infty} \|\fm\|_{\LtwoLtwo} > \|\tildef\|_{\LtwoLtwo}
	\end{equation}
	and there exists a subsequence of $\{(\gm, \fm)\}_{m\geq 1}$, denoted $\{(\gn, \fn)\}_{n \geq 1}$, for which   
	\[
	\|\fn\|_{\LtwoLtwo} \rightarrow C
	\]
	 as $n \rightarrow \infty$.
	 Then \eqref{id}  implies
	\begin{equation}
		\begin{aligned}
					&\begin{multlined}[t]\lim_{n \rightarrow \infty} \left(	J^{\pressurend}(\controltostate(\gn, \fn), \gn, \fn) - \frac{\eta}{2}\|\fn\|_{\LtwoLtwo}^2 \right)
			\end{multlined}
			\\
			=&\, \begin{multlined}[t]\frac{\nu}{2}\intTO (\controltostate(\tildeg, \tildef)-\pd)^2 \chi_{\Omega_0} \dxt+\frac{1-\nu}{2}\intO (\controltostate(\tildeg, \tildef)(T)-\pd(T))^2 \chi_{\Omega_0} \dxt \\
			+ \frac{\gamma}{2} \|\tildeg\|^2_{L^2(\LtwoG)} + \frac{\eta}{2}\left(\|\tilde{f}\|^2_{\LtwoLtwo}-C^2\right) \end{multlined} \\
			<&\, 	 \begin{multlined}[t]\frac{\nu}{2}\intTO (\controltostate(\tildeg, \tildef)-\pd)^2 \chi_{\Omega_0} \dxt+\frac{1-\nu}{2}\intO (\controltostate(\tildeg, \tildef)(T)-\pd(T))^2 \chi_{\Omega_0} \dxt \\
			 + \frac{\gamma}{2} \|\tildeg\|^2_{L^2(\LtwoG)},
			\end{multlined}
		\end{aligned}
	\end{equation}
	which is in contradiction with the lower semi-continuity property of norms. Therefore, $\fm \rightarrow \tildef$   in $L^2(0,T; \Ltwo)$.
\end{proof}
We next discuss the vanishing regularization limit. To simplify the exposition, we set $\eta=\gamma$.
\begin{proposition}
	Let the assumptions of Theorem~\ref{thm: existence opt control} hold and let $\eta=\gamma>0$ in the regularization functional \eqref{def calR}. Assume that $\pd$ is attainable, that is, there exists $(\gdagger, \fdagger)$, such that $\controltostate(\gdagger, \fdagger)=\pd$. 	Let $(\ggamma, \fgamma) \in \spacegad \times \spacefad$ be defined as an approximate minimizer in the sense that
	\begin{equation} \label{approx}
		J_\gamma^{\pdgamma}(\controltostate(\ggamma, \fgamma), \ggamma, \fgamma) \leq 	J_\gamma^{\pdgamma}(\controltostate(g, f), g, f)+\zeta(\gamma)
	\end{equation} 
holds for every $(g,f) \in \spacegad \times \spacefad$. Then for any sequence of regularizing parameters $\{\gammam\}_{m \geq 1}$ with
	\begin{equation} \label{asssumed m convergence}
		\begin{aligned}
			&	\gammam \rightarrow 0 \ \text{as}\ m \rightarrow \infty, \quad &&\zeta(\gammam) = o(\gammam), \\[1mm]
			& \|\pdgammam-\pd\|_{L^2(\Ltwo)}= o(\sqrt{\gammam}), \quad  &&\|\pdgammam(T)-\pd(T)\|_{\Ltwo}= o(\sqrt{\gammam}), 
		\end{aligned}
	\end{equation}
	the sequence $\{(\ggammam, \fgammam)\}_{m \geq 1}$ of approximate minimizers has a strongly convergent subsequence \sloppy in $L^2(0,T; \LtwoG) \times L^2(0,T; \Ltwo)$ and the limit $(\tildeg, \tildef)$ of every convergent subsequence satisfies $\controltostate(\tildeg, \tildef)=\pd$. 
\end{proposition}
\begin{proof}
	The proof is based on~\cite[Theorem 2.3]{engl1989convergence}, which, in turn, employs techniques from~\cite[Sec.\ 3]{seidman1989well}. 
		Analogously to our previous arguments, we can leverage uniform boundedness of $\{(\ggammam, \fgammam)\}_{m \geq 1} \subset \spacegad \times \spacefad$ and $\{\controltostate(\ggammam, \fgammam)\}_{m \geq 1}$, together with suitable compact embeddings, to argue that there is a subsequence and $(\tildeg, \tildef) \in \spacegad \times \spacefad$, not relabeled, such that
		\begin{equation} \label{limits gamma zero}
				\begin{aligned}
						& \ggammam \rightarrow \tildeg \ &&\text{strongly in } \ L^2(0,T; \LtwoG),\\
						& \fgammam \rightharpoonup \tildef \ &&\text{weakly in } \ L^2(0,T; \Htwo),\\
						&\controltostate(\ggammam, \fgammam) \rightarrow \controltostate(\tildeg, \tildef) \ &&\text{strongly in } \LtwoTLtwo, \\
						&\controltostate(\ggammam, \fgammam)(T) \rightarrow \controltostate(\tildeg, \tildef)(T) \ &&\text{strongly in } \Ltwo, 
					\end{aligned}
			\end{equation}
as $m \rightarrow \infty$. \\
	\indent Next, thanks to the approximate minimization property in \eqref{approx}, by choosing $(g,f)=(\gdagger, \fdagger)$, we obtain
	\begin{equation} \label{ineq vanishing reg limit}
		\begin{aligned}
			J_{\gammam}^{\pdgammam}(\controltostate(\ggammam,\fgammam), \ggammam, \fgammam) \leq 	J_{\gammam}^{\pdgammam}(\controltostate(\gdagger, \fdagger), \gdagger, \fdagger)+\zeta(\gammam).
		\end{aligned}
	\end{equation}
	The right-hand side tends of \eqref{ineq vanishing reg limit} tends to zero as $m \rightarrow 0$ due to \eqref{asssumed m convergence} and the fact that $\controltostate(\gdagger, \fdagger)=\pd$. Hence,
	\begin{equation}
		\lim_{m \rightarrow \infty} J_{\gammam}^{\pdalpham}(\controltostate(\ggammam,\fgammam), \ggammam, \fgammam) =0. 
	\end{equation}
We thus conclude that $\{\controltostate(\ggammam, \fgammam)\}_{m \geq 1}$ converges to $\pd$ strongly in $\LtwoTLtwo$. By the uniqueness of limits, it must hold that $\controltostate(\tildeg, \tildef)= \pd$. \\
	 \indent It remains to show strong convergence of (a subsequence of) $\{\fgammam\}_{m \geq 1}$ in $\LtwoTLtwo$.  Taking $(g,f)=(\tildeg, \tildef)$ in \eqref{approx} and dividing by $\gammam/2$ yields
	\begin{equation}
		\begin{aligned}
		&\|\ggammam\|^2_{L^2(\LtwoG)}+\|\fgammam\|^2_{\LtwoLtwo}\\
		 \leq&\, \begin{multlined}[t] \frac{\nu}{\gammam} \|\controltostate(\tildeg, \tildef)-\pdgammam\|^2_{L^2(L^2(\Omega_0))} +\frac{1-\nu}{\gammam} \|\controltostate(\tildeg, \tildef)(T)-\pdgammam(T)\|^2_{L^2(\Omega_0)} \\
		 	+ \|\tildeg\|^2_{L^2(\LtwoG)}+\|\tildef\|^2_{\LtwoLtwo}+ \frac{2}{\gammam}\zeta(\gammam).
		 	\end{multlined}
		 	\end{aligned}
	\end{equation}
	Thanks to \eqref{asssumed m convergence}, the right-hand side is uniformly bounded and, because $\controltostate(\tildeg, \tildef)=\pd$, we have
	\begin{equation}
		\limsup_{m \rightarrow \infty}(\|\ggammam\|^2_{L^2(\LtwoG)}+\|\fgammam\|^2_{\LtwoLtwo}) \leq \|\tildeg\|^2_{L^2(\LtwoG)}+\|\tildef\|^2_{\LtwoLtwo}.
	\end{equation}
	We wish to show that $\|\fgammam\|_{\LtwoLtwo} \rightarrow \|\tildef\|_{\LtwoLtwo}$ as $m \rightarrow \infty$, which combined with the weak convergence in \eqref{limits gamma zero} will yield the claimed strong convergence. 
	 Suppose that
	\begin{equation}
			\liminf_{m \rightarrow \infty}(\|\ggammam\|^2_{L^2(\LtwoG)}+\|\fgammam\|^2_{\LtwoLtwo}) < \|\tildeg\|^2_{L^2(\LtwoG)}+\|\tildef\|^2_{\LtwoLtwo}.
	\end{equation}
 By \eqref{limits gamma zero} and the lower semi-continuity of norms, we would then have 
 \[
 \|\tildeg\|^2_{L^2(\LtwoG)}+\|\tildef\|_{\LtwoLtwo} <\|\tildeg\|^2_{L^2(\LtwoG)}+\|\tildef\|^2_{\LtwoLtwo}.
 \]
which is a contradiction. This concludes the proof.
\end{proof}

\section{First-order necessary optimality conditions} \label{sec: optimality conditions}
In this final section, we derive the necessary optimality conditions for the distributed and Neumann boundary control problems, treating each case separately.  We begin with the analysis of the adjoint problem.

\subsection{The adjoint problem} \label{sec: adjoint}
Under the assumptions of Theorem~\ref{thm: wellp forward}, let $p$ denote the solution of the forward Westervelt problem \eqref{ibvp West}. The adjoint pressure problem in strong form is formally given by 
\begin{equation}  \label{adjoint problem}
	\left\{	\begin{aligned}
		&(1-2kp)\padjtt- \csq \Delta  \padj-b\Delta \adjointptalpha \padj =  \nu (p - \pd) \chizero\quad &&\text{ in } \Omega \times (0,T),\\[1mm]
		&\partialn \left(\csq \padj + b \, \adjointptalpha \padj \right)=0 \quad &&\text{on } \ \Gn,\\[1mm]
		&\padj(T)=0, \quad \padjt(T)=-(1-\nu)\frac{p(T)-\pd(T)}{1-2kp(T)}\chizero \quad &&\text{on }  \Omega.
	\end{aligned} \right.
\end{equation}
Due to the space-localization of the source and final data via $\chi_{\Omega_0}$, the right-hand side and final data of the adjoint problem \eqref{adjoint problem}
lack sufficient spatial regularity to apply analogous arguments to those of Proposition~\ref{prop: wellp forward linear}. Instead we interpret \eqref{adjoint problem} in a weak sense where, after time reversal, we can apply Proposition~\ref{prop: wellp forward lower}.
\begin{proposition}
	Assume the hypotheses of Theorem~\ref{thm: wellp forward} hold. Then there exists a unique
	\begin{equation}
		\begin{aligned}
			\padj \in  \left\{ \padj \in  L^\infty(0,T; \Hone) \cap  \tau \Hsubalpha(0,T; \Hone) :  \padjt \in  L^{\infty}(0,T; \Ltwo) \right\},
		\end{aligned}	
	\end{equation}
	which solves 
	\begin{equation} \label{adjoint problem weak}
		\begin{aligned}
			\begin{multlined}[t]-	\intTO ((1-2kp) \varphi)_t \,\padjt  \dxt +\intTO (\csq \nabla \padj +b \nabla \adjointptalpha \padj)\cdot \nabla \varphi \dxt \\
				=    (1-\nu)  \intO (p(T)-\pd(T)) \chizero \varphi(T) \dx+ \nu \intTO (p - \pd) \chizero  \varphi \dxt
			\end{multlined}
		\end{aligned}
	\end{equation}
	with $\padj(T)=0$, for all test functions $\varphi \in L^2(0,T; \Hone) \cap H^1(0,T; \Ltwo)$ satisfying $\varphi(0)=0$.
	Furthermore, the following estimate holds: 
	\begin{equation}
		\begin{aligned}
			\| \padj\|_{\spaceadj} \leq C(T, p) \left( (1-\nu)\|p(T)-\pd(T)\|_{L^2(\Omega_0)}	+ \nu \|p-\pd\|_{L^2(L^2(\Omega_0))}\right),
		\end{aligned}
	\end{equation}
	where
	\begin{equation}
		C(T, p) =  C_1\exp \{C_2 T (1+ \|p\|_{W^{1,\infty}(\Linf)})\}, \quad C_1, C_2>0,
	\end{equation}
	and the involved constants do not depend on $b$ or $\alpha$.
\end{proposition}
\begin{proof}
We analyze first the time reversed adjoint problem, given for $\tpadj = \tau \padj=\padj(T-t)$  by
	\begin{equation} \label{weak form adjoint reversed} 
		\begin{aligned}
			\begin{multlined}[t]-	\intTO ((1-2k\tp) \varphi)_t \,\tpadjt  \dxt +\intTO (\csq \nabla \tpadj+b \nabla \ptalpha \tpadj)\cdot \nabla \varphi \dxt \\
				=  (1-\nu) \intO  (\tp(0)-\tpd(0)) \chizero \varphi(0) \dx + \nu \intTO (\tp - \tpd) \chizero \varphi\dxt
			\end{multlined}
		\end{aligned}
	\end{equation}
	with $\tpadj(0)=0$, for all test functions $\varphi \in L^2(0,T; \Hone) \cap H^1(0,T; \Ltwo)$ satisfying $\varphi(T)=0$. 	The statement follows from Proposition~\ref{prop: wellp forward lower} together with Lemma~\ref{lemma: uniqueness} by choosing
	\begin{equation}
		\begin{aligned}
		&\fraka = 1-2k \tp, \quad   \frakl=\frakn=0, \quad F=\nu(\tp-\tpd)\chizero, \ g=0, \  u_1 =  -(1-\nu)\frac{\tp(0)-\tpd(0)}{\fraka(0)}\chizero,\\[1mm]
		&\Lambda_0(\fraka, \frakl,\frakn) =\, \exp\{CT(1+\|\fraka\|_{W^{1,\infty}(\Linf)})\}.
		\end{aligned} 
	\end{equation}
	 Indeed, by Theorem~\ref{thm: wellp forward}, we have 
\[
1-2k\tp \in W^{1,\infty}(0,T; \Linf) \cap H^2(0,T; \Lfour), \  \ 0< \ulfraka \leq  1-2k \tp  \leq \olfraka \ \text{ in } \Omega \times [0,T].
\] 
Similarly, $u_1 \in \Ltwo$ because $\tp(0)- \tpd(0)=p(T)-\pd(T) \in \Ltwo$ and $0< \ulfraka \leq \fraka(0) \leq \olfraka$ in $\Omega$. We then obtain the solution $\tpadj$ of \eqref{weak form adjoint reversed} by Proposition~\ref{prop: wellp forward lower} and of \eqref{adjoint problem weak} by setting $\padj = \tau \tpadj$.
\end{proof}

\subsection{Differentiability of the control-to-state mapping} Using the control-to-state operator \eqref{def controltostate}, consider again the reduced optimization problem
\begin{equation} \tag{\ref{reduced problem}}
	\begin{aligned}
		\text{Find } (\gstar, \fstar) \in \spacegad \times \spacefad, \text{ such that } \ j(\gstar, \fstar) = \inf_{(g,f) \in \spacegad \times \Xfad} j(g,f)
	\end{aligned}
\end{equation}
where $j= J(S(g,f),g, f)$, and $\spacegad$ and $\spacefad$ are defined in \eqref{def spaceg admissible} and \eqref{def spacef admissible}, respectively. To establish differentiability of $j$, we first show that differentiability of the control-to-state mapping $\controltostate$.
\begin{proposition} \label{prop: diff controltostate}
	Let the assumptions of Theorem~\ref{thm: wellp forward} hold. Then the control-to-state mapping $\controltostate$ defined in \eqref{def controltostate} is directionally differentiable 
	in the following sense. Let $(h, \phi) \in \spacegad \times \spacefad$ be such that $(g+ \eps h, f+\eps \phi) \in \spacegad \times \spacefad$ for $\eps \in [0, \bar{\eps})$. Define
	\begin{equation}
	\zeps = \dfrac{\peps-p}{\eps} ,\quad p=\controltostate(g, f),\ \text{ and } \ \peps= \controltostate(g+\eps h, f+\eps \phi).
	\end{equation}
Then there exists a subsequence $\{\zepsn\}_{n \in \N}$ of $\{\zeps\}_{\eps \in (0, \bar{\eps})}$, such that
	\begin{equation} \label{limits nonlinear_problem differentiability thm} 
	\begin{aligned}
		\zepsn &\overset{\ast}{\rightharpoonup} z \quad  && \text{in} \quad L^\infty(0,T; \Hone), \quad
		\zepsnt\overset{\ast}{\rightharpoonup} z_t \quad   \text{in} \quad  L^\infty(0,T; \Ltwo), \\
		\ptalpha \zepsn &{\rightharpoonup}\, \ptalpha z \quad  && \text{in} \quad L^2(0,T; \Hone),\\
	\end{aligned}
\end{equation}
	as $n \rightarrow \infty$. The limit $z$ satisfies 	
	\begin{equation} \label{limit nonlinear_problem differentiability thm}  
		\begin{aligned}
			\begin{multlined}[t]-	\intTO ((1-2k p)z)_{t}\, \varphi_t \dxt +\intTO (\csq \nabla z +b \nabla \ptalpha z)\cdot \nabla \varphi \dxt \\
			- \intTG (c^2 h+ b \ptalpha h) \phi \dGt   
				=   \intTO \phi \varphi \dxt
			\end{multlined}
		\end{aligned}
	\end{equation}
with $z(0)=0$, for all $\varphi \in L^2(0,T; \Hone) \cap H^1(0,T; \Ltwo)$ with $\varphi(T)=0$.
	
\end{proposition}

\begin{proof}
Thanks to Theorem~\ref{thm: wellp forward}, the family $\{\peps\}_{\eps \geq 0}$ is uniformly bounded:
\begin{equation} \label{bound peps}
	\|\peps\|_{\Xp} \lesssimT \|g+\eps h\|_{\spaceg} + \|f+ \eps \phi\|_{\spacef} \lesssimT (1+ \bar{\eps})(L_1+L_2).
\end{equation}
 The quotient $\zeps = \dfrac{\peps-p}{\eps} $ satisfies 
	\begin{equation} \label{limit nonlinear_problem differentiability thm weak}  
	\begin{aligned}
		\begin{multlined}[t]-	\intTO (1-2kp)  \zeps_t \varphi_t \dxt +\intTO 2k \zeps  \peps_t \varphi_t \dxt \\+\intTO (\csq \nabla \zeps +b \nabla \ptalpha \zeps)\cdot \nabla \varphi \dxt 
			- \intTG (c^2 h+ b \ptalpha h) \varphi \dGt    \\
			=   \intTO \phi \varphi \dxt
		\end{multlined}
	\end{aligned}
\end{equation}
	for all $\varphi \in L^2(0,T; \Hone) \cap H^1(0,T; \Ltwo)$ with $\varphi(T)=0$.	We now apply estimate \eqref{final est lower}  from Proposition~\ref{prop: wellp forward lower} by choosing
	\begin{equation}
		\frakm = 1-2kp,  \quad  \frakl = -2k \pepst, \quad  \frakn = -2k \pepstt, \quad F=\phi, \quad g = h, \ \text{and} \ u_1=0.
	\end{equation}
This yields
	\begin{equation} \label{est}
		\begin{aligned}
		\|\zeps\|_{\Xlow} \lesssim \Lambda_0(\fraka, \frakl, \frakn)\left( \|h\|_{H^2(\HneghalfG)}+\|\phi\|_{\LtwoLtwo}\right),
		\end{aligned}
	\end{equation}
	where $\Lambda_0$ is defined in \eqref{def Lambda_0}. Thanks to Theorem~\ref{thm: wellp forward} with the assumption $\|k\|_{\Xk} \leq \delta$, we obtain
	\begin{equation} 
	\begin{aligned}
		\Lambda_0(\fraka, \frakl, \frakn)= 	&\,	\exp\left\{CT(	1+\|\fraka\|_{W^{1, \infty}(\Linf)}+\|\frakl\|_{L^\infty(\Linf)}+\|\frakn\|^2_{L^2(\Lthree)})\right\}\\
			\lesssim&\,	\exp\left\{CT\left(	1+\delta \left(\|p\|_{W^{1, \infty}(\Linf)}+\|\pepst\|_{L^\infty(\Linf)}+\|\pepstt\|^2_{L^2(\Hone)} \right)\right)\right\}\\
				\lesssim&\,	\exp\left\{CT\left(	1+\delta \|\peps\|_{\Xp}  + \delta \|\peps\|^2_{\Xp}\right)\right\}.
	\end{aligned}
\end{equation}
Together with \eqref{bound peps}, this shows that $\Lambda_0$ is uniformly bounded.  Therefore, estimate \eqref{est} implies that $\{\opeps\}_{\eps}$ is uniformly bounded in $\Xlow $, defined in \eqref{def Xlow}. Consequently, there exists a subsequence $\{z^{\eps_n}\}_{n \in \N}$, such that the convergences in
		\eqref{limits nonlinear_problem differentiability thm} hold. 
		To pass to the limit in 
\eqref{limit nonlinear_problem differentiability thm weak} 	as $\epsn \rightarrow 0$, we use the weak convergence 
		\begin{equation}
			\begin{aligned}
			-	(1-2k p) \zepsnt + 2k \zeps  \peps_t\ \rightharpoonup \ 	-(1-2k p) z_t+ 2k z  \pt= -((1-2k p)z)_{t}
			\end{aligned}
		\end{equation}
		in $L^2(0,T; \Ltwo)$, which follows analogously to \eqref{product limits} and to~\cite[p.\ 749]{clason2009boundary}.  Passing to the limit in \eqref{limit nonlinear_problem differentiability thm weak}  as $\epsn \rightarrow 0$ therefore yields \eqref{limit nonlinear_problem differentiability thm}, as claimed.
		\end{proof}
		\subsection{Derivative of the reduced objective} Next, we compute the derivative of the reduced objective. Using the chain rule and recalling \eqref{def j reduced}, together with the fact that  $p=\controltostate(g,f)$, we obtain
		\begin{equation} \label{j prime}
			\begin{aligned}
				j'(g, f; h, \phi)=&\,\begin{multlined}[t] D_f J(p, g, f; \phi) + D_g J(p, g, f; h) + D_p J(p,g, f; \controltostate'(g, f;h, \phi)),
				\end{multlined}
			\end{aligned}
		\end{equation}
		where
		\begin{equation} \label{derivatives f g}
			\begin{aligned}
				D_f J(p, g, f; \phi) =	 \eta \intTO f \phi \dxt,\quad
				D_g J(p, g, f; h) 
				= \gamma \intT \int_{\Gamma} g h \dGt, 
			\end{aligned}
		\end{equation}
		and
		\begin{equation}
			\begin{aligned}
				&D_p J(p,g, f; \controltostate'(g,f; h, \phi))\\
				=&\, \nu \intTO (p-\pd) \chizero \controltostate'(g,f;h, \phi) \dxt+(1-\nu)\intO (p(T)-\pd(T))\chizero \controltostate'(g,f;h, \phi)(T)\dx.
			\end{aligned}
		\end{equation}
		We now use the adjoint problem to express these terms without the explicit appearance of $\controltostate'$.  Testing \eqref{adjoint problem weak} with $z = \controltostate'(g,f;h, \phi)$ yields
	\begin{equation} 
	\begin{aligned}
			    &(1-\nu)  \intO (p(T)-\pd(T)) \chizero z(T) \dx+ \nu \intTO (p - \pd) \chizero z \dxt \\
			   =&\, -	\intTO((1-2k p)z)_{t}\, \padjt   \dxt +\intTO (\csq \nabla \padj + b \nabla \adjointptalpha \padj)\cdot \nabla z \dxt\\
	=&\, 			\begin{multlined}[t] -	\intTO((1-2k p)z)_{t}\, \padjt   \dxt +\intTO (\csq \nabla z +b \nabla \ptalpha z)\cdot \nabla \padj \dxt.
			\end{multlined}
		\end{aligned}
	\end{equation}
Testing the weak form \eqref{limit nonlinear_problem differentiability thm}  solved by $z$  with $\padj$ gives an equivalent form:
			\begin{equation}\label{D_p_S 1}
		\begin{aligned}
			D_p J(p,g, f; \controltostate'(g, f;h, \phi))
			= \intT \int_{\Gamma} (\csq h +b\ptalpha h) \padj \dGt + \intTO \phi \padj \dxt.
		\end{aligned}
	\end{equation}
	 Combining \eqref{D_p_S 1} with \eqref{derivatives f g}, we finally obtain
		\begin{equation} \label{derivative}
			\begin{aligned}
				j'(g,f;h, \phi) =&\, \begin{multlined}[t] \intT \int_{\Gamma} (\csq h +b \ptalpha h) \padj \dGt + \intTO \phi \padj \dxt\\ \hspace*{2cm}
					+\gamma \intT \intG gh \dGt + \eta \intTO f \phi \dxt.
				\end{multlined}
			\end{aligned}
		\end{equation}
	We now state the necessary optimality conditions. From this point onward, we treat the distributed and boundary control cases separately.
		
		\subsection{Optimality conditions for the Numann boundary control problem} We focus now on
	 the reduced boundary control problem
		\begin{equation} \label{reduced problem only g}
			\begin{aligned}
			\fbox{$	\text{Find } \gstar \in  \spacegad, \text{ such that } \ j(\gstar) = \displaystyle \inf_{g \in  \spacegad} j( g)$}
			\end{aligned}
		\end{equation}
		with the objective $j(g) = J(\controltostate(g), g)$, where
		\begin{equation}
			\begin{aligned} 
				J(p, f) =&\, \begin{multlined}[t] \frac{\nu}{2}\| p-\pd\|^2_{L^2(L^2(\Omega_0))}+\frac{1-\nu}{2}\| p(T)-\pd(T) \|^2_{L^2(\Omega_0)} 
					+\frac{\gamma}{2}\|g\|^2_{L^2(L^2(\Gamma))}.
				\end{multlined}
			\end{aligned}
		\end{equation}
	For $h\in \spacegad$ (cf.\ \eqref{def spaceg admissible}), we have the identity
		\begin{equation}
			\begin{aligned}
			b 	\intT \int_{\Gamma} \ptalpha h \, \padj \dGt  =&\, -b \intT \int_{\Gamma} h \adjointptalpha \padj \dGt.
			\end{aligned}
		\end{equation}
		Thus in the setting of exerting control only via the boundary excitation, we find that
		\begin{equation} \label{derivative wrt g}
			\begin{aligned}
				j'(g;h)=&\, \intT \int_{\Gamma} (\csq \padj-b \adjointptalpha \padj +\gamma g)h  \dGt.
			\end{aligned}
		\end{equation}
		The first-order necessary optimality conditions are then immediately  obtained by employing~\cite[Lemma 2.21]{troltzsch2024optimal}; see also~\cite[Theorem 9.2]{manzoni2021optimal}.	
		\begin{theorem}[Necessary optimality conditions for boundary control] \label{thm: opt cond boundary control}
		Let the assumptions of Theorem~\ref{thm: wellp forward} hold.	Let $\tildeg$ be a local solution of the reduced minimization problem
			\begin{equation} \label{reduced problem only g}
				\begin{aligned}
					\min_{g \in \spaceg}   j(g)  \quad  \textup{s.t.}  \quad  g \in \spacegad,
				\end{aligned}
			\end{equation} 
			where $\spacegad$ is defined in \eqref{def spaceg admissible}. Then $\tildeg$ satisfies the following variational inequality:
			\begin{equation}
				\intT \int_{\Gamma} (\csq \padj-b \adjointptalpha \padj +
				\gamma \tildeg )(\tildeg -g) \dGt \geq 0,
			\end{equation}
			where $p$ solves \eqref{ibvp West} and $\padj$ the adjoint problem in the weak sense \eqref{adjoint problem weak}.
		\end{theorem}
		\subsection{Optimality conditions for the distributed control problem} Secondly, we focus on the reduced distributed control problem
		\begin{equation} \label{reduced problem only g}
			\begin{aligned}
			\fbox{$	\text{Find } \fstar \in  \Xfad, \text{ such that } \ j(\fstar) =\displaystyle  \inf_{f \in  \Xfad} j( f)$}
			\end{aligned}
		\end{equation}
		with $j(f) = J(\controltostate(f), f)$, where here
				\begin{equation}
			\begin{aligned} 
				J(p, f) =&\, \begin{multlined}[t] \frac{\nu}{2}\| p-\pd\|^2_{L^2(L^2(\Omega_0))}+\frac{1-\nu}{2}\| p(T)-\pd(T) \|^2_{L^2(\Omega_0)} 
					+\frac{\eta}{2}\|f\|^2_{L^2(L^2(\Omega))}.
				\end{multlined}
			\end{aligned}
		\end{equation}
		In this setting, analogously to \eqref{derivative}, we have
		\begin{equation}
			j'(f; \phi) = \intTO (\padj + \eta  f) \phi \dxt.
		\end{equation}
		Similarly to Theorem~\ref{thm: opt cond boundary control}, we have the following first-order necessary optimality conditions.
		\begin{theorem}[Necessary optimality conditions for distributed control] \label{thm: opt cond distributed control}
				Let the assumptions of Theorem~\ref{thm: wellp forward} hold.	Let $\tildef \in \spacefad$ be a solution of the reduced minimization problem
			\begin{equation} \label{reduced problem only f}
				\begin{aligned}
					\text{Find } \tildef \in \spacefad, \text{ such that } \ j(\tildef) = \inf_{f \in \spacefad} j(f),
				\end{aligned}
			\end{equation} 
			where $\spacef$ is defined in \eqref{def spacef admissible}.
			Then $\tildef$ satisfies the following variational inequality:
			\begin{equation}
				\intTO (\padj + \eta f )(\tildef -f) \dxt \geq 0 \quad \forall f \in \spacefad,
			\end{equation}
			where $p$ solves the state problem \eqref{ibvp West} and $\padj$ the adjoint problem in the weak sense \eqref{adjoint problem weak}.
		\end{theorem}


\section*{Conclusion}
The aim of this work was to deepen the understanding of how nonlinear \sloppy acoustic waves can be controlled in complex propagation media.  To this end, we analyzed distributed and boundary optimal control problems subject to the Westervelt equation with time-fractional attenuation, a model relevant ultrasound propagation through biological media and, consequently, for various medical applications.  First, we extended the existing well-posedness theory for the Westervelt equation with nonlocal damping to the setting of inhomogeneous Neumann data. This required constructing a suitable extension of regularized boundary data first. Second, we established the existence of globally optimal controls and analyzed the stability of the optimization problem with respect to perturbations in the targeted pressure distribution and the vanishing regularization parameters. Third, we showed that the adjoint problem, which has state-dependent coefficients, admits a solution even though it does not share the regularity level of the state equation. This allowed us to prove differentiability of the control-to-state mapping in a reduced regularity setting. Finally, these results justified the derivation of adjoint-based first-order optimality conditions for the control problem. \\
\indent Natural directions for future research include developing and analyzing efficient numerical optimization algorithms in this context, as well as building upon the present theoretical framework to incorporate Westervelt-based systems as state constraints, such as the wave-heat systems motivated by ultrasound heating phenomena~\cite{careaga2025westervelt}. Additionally, as the analysis of the state problem in Theorem~\ref{thm: wellp forward} is uniform in $b$ and $\alpha$, this could be used as a basis for investigating the limiting behavior of minimizers as $b\searrow 0$ and $\alpha \nearrow 1$ and establishing a connection to problems constrained by the inviscid and strongly damped Westervelt equations, respectively.

\begin{appendices}
	\section{Details on the derivation of energy estimates} \label{appendix: energy estimates}
	We provide here the missing details of the second and third testing steps within the Faedo--Galerkin procedure in the proof of Theorem~\ref{thm: wellp forward}.
\subsection{Derivation of estimate \eqref{step 2: energy_ineq}} \label{appendix: energy est 1}	Testing the semi-discrete equation with $-\Delta \wregmt$ yields
	\begin{equation} \label{identity_1 Step 2}
		\begin{aligned} 
			&(\fraka \wregmtt -\csq \Delta \wregm- b \Delta \ptalpha \wregm,\,   -\Delta \wregmt)_{\Ltwo} \\
			=&\, ( \FGreg-\frakl \wregmt-\frakn \wregm, -\Delta \wregmt)_{\Ltwo}.
		\end{aligned}
	\end{equation}
Since $\wregmt$ satisfies homogeneous Neumann data, integration by parts yields the identity
	\begin{equation}
		\begin{aligned}
			&(\fraka \wregmtt ,  -\Delta \wregmt)_{\Ltwo} \\
			=&\, 	(\nabla(\fraka \wregmtt) ,  \nabla \wregmt)_{\Ltwo}-	\intG \fraka \wregmtt  \frac{\partial \wregmt}{\partial n}\dG \\
			=&\, \frac12 \ddt	(\fraka \nabla  \wregmt ,  \nabla \wregmt)_{\Ltwo}-\frac12 	(\frakat \nabla  \wregmt ,  \nabla \wregmt)_{\Ltwo} +	(\wregmtt \nabla \fraka,  \nabla \wregmt)_{\Ltwo}.
		\end{aligned}
	\end{equation}
The last two terms above are placed on the right-hand side and, after integrating in time, estimated by
	\begin{equation}
		\begin{aligned}
			&\frac12 \intt  	(\frakat \nabla  \wregmt ,  \nabla \wregmt)_{\Ltwo} -	(\wregmtt \nabla \fraka,  \nabla \wregmt)_{\Ltwo} \ds \\
			\lesssim&\, \intt\|\frakat\|_{\Linf}\|\nabla \wregmt\|^2_{\Ltwo}+ \intt \|\nabla \fraka\|_{\Linf}\|\wregmtt\|_{\Ltwo}\|\nabla \wregmt\|_{\Ltwo}.
		\end{aligned}
	\end{equation}
The terms involving $\frakl$ and $\frakn$ are controlled via H\"older's and Young’s inequalities:
	\begin{equation}
		\begin{aligned}
			&\intt	 ( -\frakl \wregmt-\frakn \wregm, -\Delta \wregmt)_{\Ltwo} \ds \\
			\leq& \frac12 \|-\frakl \wregmt-\frakn \wregm\|_{\LtwotLtwo}^2+\frac12 \|\Delta \wregmt\|^2_{\LtwotLtwo}.
		\end{aligned}
	\end{equation}
Similarly,
	\begin{equation}
		\begin{aligned}
		\intt	 ( \FGreg, -\Delta \wregmt)_{\Ltwo} \ds \lesssim&\, \intt \|\FG \|_{\Ltwo}\|\Delta \wregmt\|_{\Ltwo}\ds.
		\end{aligned}
	\end{equation}
Integrating \eqref{identity_1 Step 2} over $(0,t)$, inserting these estimates, and invoking the coercivity estimate \eqref{coercivity ineq}, we obtain
	\begin{equation} \label{energy_ineq_1 Step 2}
		\begin{aligned}
			&\frac12 	\|\sqrt{\fraka(t)}\nabla  \wregmt(t)\|^2_{\Ltwo}+	 \csqhalf(T) 	\|\D \wregm(t)\|^2_{\Ltwo} + b
			\Calpha(T)  \|\Delta \ptalpha  \wregm \|^2_{\LtwotLtwo} \\
			\lesssim&\,\begin{multlined}[t]  \|\frakat\|_{\LinfLinf}\|\nabla \wregmt\|^2_{\LtwotLtwo}+  \|\nabla \fraka\|_{\LinfLinf}\|\wregmtt\|_{\Ltwo}\|\nabla \wregmt\|_{\LtwotLtwo}\\
				+  \|\FGreg \|_{\LtwoLtwo}^2+\|\Delta \wregmt\|_{\LtwotLtwo}^2\\
				+\|\frakl\|^2_{\LinfLinf}\|\wregmt\|^2_{\LtwotLtwo}+\|\frakn\|^2_{\LinfLinf}\|\wregm\|^2_{\LtwotLtwo}.
			\end{multlined}
		\end{aligned}
	\end{equation}
To estimate $\|\wregmtt\|_{\LtwotLtwo}$, we can rely on the identity
	\begin{equation} \label{id  ptt}
		\begin{aligned}
			(\frakm	 \wregmtt, \wregmtt)_{\Ltwo}
			=&\, \begin{multlined}[t]  (\csq  \Delta \wregm+b  \Delta \ptalpha \wregm - \frakl \wregmt-\frakn \wregm+ \FGreg,  \wregmtt)_{\Ltwo},
			\end{multlined}  
		\end{aligned}
	\end{equation}
which can be seen as an additional testing by $\wregmtt$, from which we obtain by recalling \eqref{nondegeneracy assumption frakm frakm} 
	\begin{equation}  \label{ineq0} 
	\begin{aligned}
		\|\wregmtt\|_{\LtwotLtwo} \lesssim&\, \|  \csq \Delta w+b \Delta \ptalpha \wregm - \frakl \wregmt-\frakn \wregm+ \FGreg\|_{L^2_t(\Ltwo)}\\
		\lesssim&\, \begin{multlined}[t]
			\|   \Delta w\|_{\LtwotLtwo}+b \|{\Delta \ptalpha \wregm}\|_{\LtwotLtwo} +\| \frakl\|_{\LinfLinf}\| \wregmt\|_{\LtwotLtwo}\\+\|\frakn\|_{\LinfLinf} \|\wregm\|_{\LtwotLtwo} +\| \FGreg\|_{L^2_t(\Ltwo)}.
			\end{multlined}
	\end{aligned}	
	\end{equation}
	Thus \eqref{energy_ineq_1 Step 2} becomes
	\begin{equation} \label{energy_ineq_2 Step 2}
		\begin{aligned}
			&\frac12 	\|\sqrt{\fraka(t)}\nabla  \wregmt(t)\|^2_{\Ltwo}+	 \csqhalf 	\|\D \wregm(t)\|^2_{\Ltwo}+ b
			\ulC(T)  \|\Delta \ptalpha \wregm \|^2_{\LtwotLtwo} \\
			\lesssim&\,\begin{multlined}[t]  \calLfraka\|\nabla \wregmt\|^2_{\LtwotLtwo}+  \calLfraka^2 \Bigl\{ \|   \Delta w\|_{\LtwotLtwo}+b \|{\Delta \ptalpha \wregm}\|_{\LtwotLtwo} \\+\| \wregmt\|_{\LtwotLtwo}+\| \wregm\|_{\LtwotLtwo}  +\| \FGreg\|_{L^2_t(\Ltwo)}\Bigr\}\|\nabla \wregmt\|_{\LtwotLtwo}\\ +	 \|\FGreg \|^2_{\LtwoLtwo}+\|\Delta \wregmt\|_{\LtwotLtwo}^2+\calLfraka^2 \|\wregmt\|^2_{\LtwotLtwo}\\+\calLfraka^2\|\wregm\|^2_{\LtwotLtwo}.
			\end{multlined}
		\end{aligned}
	\end{equation}
Young's $\varepsilon$-inequality with $\varepsilon>0$ sufficiently small  allows the  resulting $\eps \|{\Delta \ptalpha \wregm}\|_{\LtwotLtwo} ^2$ term on the right to be absorbed into the left-hand side of \eqref{energy_ineq_2 Step 2}:
	\begin{equation}
			\begin{aligned}
					&\calLfraka^2 b	\|\Delta \ptalpha \wregm\|_{\LtwotLtwo} \|\nabla \wregmt\|_{\LtwotLtwo} \\
					\leq&\, \eps b	\|\Delta \ptalpha \wregm\|_{\LtwotLtwo}^2+ \frac{1}{4\eps} \bar{b}\calLfraka^4 \|\nabla \wregmt\|_{\LtwotLtwo}^2,
				\end{aligned}
		\end{equation}
 which leads to \eqref{step 2: energy_ineq}.

\subsection{Derivation of estimate  \eqref{step 3: energy_ineq}} \label{appendix: energy est 2}  We provide here the missing  details for testing step 3 in the proof of Theorem~\ref{thm: wellp forward}.	 To treat the $\fraka$ term, we use the identity
\begin{equation}\label{Identity_m}
	\begin{aligned}
		&(\Delta(\fraka \wregmtt), \Delta \wregmt)_{\Ltwo} \\
		=&\,  (\fraka\Delta \wregmtt, \Delta \wregmt)_{\Ltwo}+(\wregmtt\, \Delta \fraka +2\nabla \wregmtt \cdot\nabla\fraka, \Delta \wregmt)_{\Ltwo} \\
		=&\, \begin{multlined}[t] \frac12\ddt(	\fraka\Delta \wregmt, \Delta \wregmt)_{\Ltwo}-\frac12(\frakm_t\D \wregmt, \Delta \wregmt)_{\Ltwo}\\ +(\wregmtt \, \Delta \fraka+2\nabla \wregmtt \cdot\nabla\fraka, \Delta \wregmt)_{\Ltwo}. 
		\end{multlined}
	\end{aligned}
\end{equation}
From \eqref{bases_Lapl}, we have $\frac{\partial}{\partial n}\Delta \wregm =0$ on $\Gamma$. Hence, integration by parts yields
\begin{equation}
	\begin{aligned}
		-	(\csq \Delta^2 \wregm, \Delta \wregmt)_{L^2}
		=&\, 	\csq ( \nabla  \Delta \wregm, \nabla \Delta \wregmt)_{\Ltwo} 
		=	\frac12\csq \ddt\|	\nabla\D \wregm\|_{\Ltwo}^2.
	\end{aligned}
\end{equation}
Integrating by parts, using the fact that $\frac{\partial}{\partial n}\Delta \ptalpha \wregm =0$, and employing the coercivity estimate \eqref{coercivity ineq} , we obtain
\begin{equation}
	\begin{aligned}
		-	\intt (b {\Delta^2 \ptalpha \wregm}, \Delta \wregmt)_{\Ltwo}\ds
		=&\, b \intt ( \nabla \ptalpha \Delta \wregm, \nabla \Delta \wregmt)_{\Ltwo} \ds \\
		\geq&\, b {\Calpha  \| \nabla \D \ptalpha \wregm \|^2_{\LtwotLtwo}}. 
	\end{aligned}
\end{equation}
Substituting these identities and lower bounds in \eqref{step 3: identity_1} yields
\begin{equation} \label{step 3: ineq 1}
	\begin{aligned} 
		&\begin{multlined}[t]\frac12\ddt(	\frakm\D \wregmt, \Delta \wregmt)_{\Ltwo}+ \frac12\csq \ddt\|	\nabla\D \wregm\|_{\Ltwo}^2+b  \Calpha  \|\nabla \D \ptalpha \wregm\|^2_{\LtwotLtwo} 
		\end{multlined}  \\
		\leq &\,\begin{multlined}[t]
			\frac12(\frakm_t\D \wregmt, \Delta \wregmt)_{\Ltwo}-(\wregmtt \, \D\frakm+2\nabla \wregmtt \cdot\nabla\frakm, \Delta \wregmt)_{\Ltwo}
			\\ +(\Delta \FGreg-\Delta(\frakl \wregmt)-\Delta(\frakn \wregm), \Delta \wregmt)_{L^2}.
		\end{multlined}
	\end{aligned}
\end{equation}
We treat the $\frakl$ and $\frakn$ terms on the right-hand side as follows:
\begin{equation}\label{l_identity}
	\begin{aligned}
		&- (\D (\frakl \wregmt), \Delta \wregmt )_{L^2}- (\D (\frakn \wregm), \Delta \wregmt )_{\Ltwo}\\
		=&\, -( \frakl\D \wregmt +2 \nabla \frakl \cdot \nabla \wregmt +  \wregmt \D \frakl+\frakn\D \wregm +2 \nabla \frakn \cdot \nabla \wregm +  \wregm \D \frakn, \Delta \wregmt)_{\Ltwo}.
	\end{aligned}
\end{equation}
We estimate the $\frakl$-dependent terms by
\begin{equation}
	\begin{aligned}
		&\intt ( \frakl\D \wregmt +2 \nabla \frakl \cdot \nabla \wregmt +  \wregmt \D \frakl, \Delta \wregmt )_{L^2} \ds \\
		\lesssim&\,\begin{multlined}[t]  \|\frakl\|_{\LinfLinf} \|\Delta \wregmt\|^2_{\LtwotLtwo} + \|\nabla \frakl\|_{\LinfLthree} \|\nabla \wregmt\|_{L^2_t(\Lsix)} \|\Delta \wregmt\|_{\LtwotLtwo} \\+ \|\D \frakl\|_{\LinfLtwo} \|\wregmt\|_{L^2_t(\Linf)}\|\Delta \wregmt\|_{\LtwotLtwo}
		\end{multlined} \\
		\lesssim&\, \|\frakl\|_{\Xfrakl} \| \wregmt\|^2_{\LtwotHtwo},
	\end{aligned}
\end{equation}
where we have used the embedding $\Htwo \hookrightarrow \Linf$ in the last step.  Similarly,
\begin{equation}
	\begin{aligned}
		&\intt ( \frakn \D \wregm +2 \nabla \frakn \cdot \nabla \wregm +  \wregm \D \frakn, \Delta \wregmt )_{L^2} \ds \\
		\lesssim&\, \|\frakn\|_{\Xfrakn} \| \wregmt\|_{\LtwotHtwo}\|w\|_{\LtwotHtwo}.
	\end{aligned}
\end{equation}
After integrating \eqref{step 3: identity_1} over $(0,t)$ for $t \in (0,T)$ and inserting the derived estimates, we obtain
								\begin{equation} \label{energy_ineq_1}
		\begin{aligned}
			&\begin{multlined}[t]  \frac12\nLtwo{\sqrt{\fraka} \Delta \wregmt}^2 \Big \vert_0^t + \csqhalf\nLtwo{\nabla\D \wregm}^2 \Big \vert_0^t + b
				\ulC(T) \|\nabla \D \ptalpha \wregm \|^2_{\LtwotLtwo} 
			\end{multlined} \\
			\leq&\,\begin{multlined}[t] 
				\frac12 \int_0^t (\frakat\D \wregmt, \Delta \wregmt)_{L^2}\ds+ \intt(\wregmtt \, \D\frakm+2\nabla \wregmtt \cdot\nabla\frakm, \Delta \wregmt)_{L^2} \ds \\
				+ \|\Delta F\|_{\LtwoLtwo}^2+ \|\Delta \wregmt\|^2_{\LtwotLtwo} +\calLfraka\|\greg\|_{\spaceglower}  \|\Delta \wregmt\|_{\LtwotLtwo}  
				\\	+  \calLfraka \| \wregmt\|^2_{\LtwotHtwo}+\calLfraka \| \wregm\|_{\LtwotHtwo}\| \wregmt\|_{\LtwotHtwo}.
			\end{multlined}
		\end{aligned}
	\end{equation}
The first term on the right–hand side is bounded using
	\begin{equation} \label{est_aaa_t}
		\begin{aligned}
			\frac12 \int_0^t (\frakat\D \wregmt, \Delta \wregmt)_{L^2}\ds
			\leq&\, \frac12 \intt \|\frakm_t\|_{\Linf}  \|\Delta \wregmt\|^2_{\Ltwo} \ds.
		\end{aligned}
	\end{equation} 
To estimate the $\wregmtt$ terms on the right-hand side of \eqref{energy_ineq_1}, we first use H\"older's and Young's inequalities:
	\begin{equation} \label{est_utteps}
		\begin{aligned}
			&\intt(\wregmtt \, \D\fraka+2\nabla \wregmtt \cdot\nabla\frakm, \Delta \wregmt)_{L^2} \ds \\
			\lesssim&\,  \left(\|\wregmtt\|_{L^2(\Lsix)}\|\Delta \frakm\|_{L^\infty(\Lthree)} + \|\nabla \wregmtt\|_{L^2(\Ltwo)}\|\nabla \frakm\|_{L^\infty(\Linf)} \right)\|\Delta \wregmt\|_{L^2(\Ltwo)}.
		\end{aligned}
	\end{equation}
	By the Sobolev embeddings $\Hone \hookrightarrow \Lthree$ and $\Htwo \hookrightarrow \Linf$, we obtain
	\begin{equation}
		\begin{aligned}
			\|\Delta \frakm\|_{L^\infty(\Lthree)} +	\|\nabla \frakm\|_{L^\infty(\Linf)} \lesssim 	\|\Delta \frakm\|_{L^\infty(\Hone)} +	\|\nabla \frakm\|_{L^\infty(\Htwo)} \lesssim \|\frakm\|_{\Xfrakm}.
		\end{aligned}
	\end{equation}
Hence,
	\begin{equation} \label{est_utt_interim}
		\begin{aligned}
			&\intt(\wregmtt \, \D\fraka+2\nabla \wregmtt \cdot\nabla\fraka, \Delta \wregmt)_{L^2} \ds \\
			\leq&\,  \calLfraka \|\wregmtt\|_{L^2(\Hone)} \|\Delta \wregmt\|_{L^2(\Ltwo)}.
		\end{aligned}
	\end{equation}
To bound $\|\wregmtt\|_{L^2(\Hone)}$, we note that
	\begin{equation}
		\begin{aligned}
			\|\wregmtt\|_{L^2(\Hone)}  \lesssim \|  \wregmtt\|_{L^2(\Ltwo)}+\| \nabla \wregmtt\|_{L^2(\Ltwo)}.
		\end{aligned}
	\end{equation}
We estimated the first term on the right in \eqref{ineq0}. For the second one, we use the identity
	\begin{equation} \label{id nabla ptt}
		\begin{aligned}
		&	(\frakm	\nabla \wregmtt+ \nabla \frakm \wregmtt, \nabla \wregmtt)_{\Ltwo}\\
			=&\, \begin{multlined}[t]  (\csq \nabla \Delta \wregm+b  \nabla \Delta \ptalpha \wregm-\nabla [\frakl \wregmt+\frakn \wregm]+\nabla \FGreg, \nabla \wregmtt)_{\Ltwo},
			\end{multlined}
		\end{aligned}
	\end{equation}
obtained by applying $\nabla$ to the semi-discrete PDE and testing with $\nabla \wregmtt$. We then have
	\begin{equation} \label{ineq1}
		\begin{aligned}
			&\|\nabla  \wregmtt\|_{\LtwotLtwo} \\
			\lesssim&\,  \begin{multlined}[t] 
				\|\nabla \frakm\|_{L^\infty(\Linf)}\|  \csq\Delta \wregm+b  \Delta \ptalpha \wregm- \frakl \wregmt-\frakn \wregm+ \FGreg\|_{\LtwotLtwo}\\
				+ \| b \nabla \Delta \ptalpha \wregm+\csq\nabla \Delta \wregm-  \nabla [\frakl \wregmt]-  \nabla [\frakn \wregm]+\nabla \FGreg\|_{\LtwotLtwo}.
			\end{multlined}
		\end{aligned}
	\end{equation}
Estimating the right-hand side terms in \eqref{ineq1} and incorporating \eqref{ineq0} yields
	\begin{equation} \label{bound H1 wtt}
		\begin{aligned}
			&\|  \wregmtt\|_{\LtwotHone} \\
			\lesssim&\,  \begin{multlined}[t] 
			\calLfraka \Big (b \|  \Delta \ptalpha \wregm\|_{\LtwotHone} + \|  \Delta \wregm\|_{\LtwotHone}+ 
				\| \frakl\|_{\Xfrakl} \|\wregmt\|_{L^2(\Hone)}\\+	\| \frakn\|_{\Xfrakn} \|\wregm\|_{\LtwotHone}+ 
				\| \FGreg\|_{L^2(\Hone)} \Bigr)
			\end{multlined}\\
			\lesssim&\,  \begin{multlined}[t] 
				\calLfraka^2\Big ( b\|  \Delta \ptalpha \wregm\|_{\LtwotHone} + \|   \wregm\|_{\LtwotHtwo}+ 
				\|\wregmt\|_{L^2(\Hone)}+ 
				\| \FGreg\|_{L^2(\Hone)} \Bigr).
			\end{multlined}
		\end{aligned}
	\end{equation}
Returning to \eqref{est_utt_interim} and applying \eqref{bound H1 wtt} gives
	\begin{equation} \label{est_utt_final}
		\begin{aligned}
			&\intt(\wregmtt \, \D\fraka+2\nabla \wregmtt \cdot\nabla\frakm, \Delta \wregmt)_{L^2} \ds \\
			\lesssim&\,  \begin{multlined}[t]   \calLfraka^3 \Big\{ \|   \wregm\|_{\LtwotHtwo}+   \| \wregmt \|_{\LtwotHone}
				+b \| \Delta \ptalpha \wregm\|_{\LtwotHone} \\+\| \FGreg\|_{L^2(\Hone)}\Big\} \|\Delta \wregmt\|_{\LtwotLtwo}.
			\end{multlined}
		\end{aligned}
	\end{equation}
	Analogously to before, Young's inequality applied for any $\eps>0$ yields:
	\begin{equation} \label{est_Keps_gamma}
		\begin{aligned} 
			&  \calLfraka^3 b \|   \Delta \ptalpha \wregm\|_{L^2_t(\Hone)}  (\|\Delta \wregmt\|_{L^2_t(\Ltwo)}+\|\nabla \wregmt\|_{L^2_t(\Ltwo)})\\
			\leq&\,  \eps b	\|   \Delta \ptalpha \wregm\|^2_{L^2_t(\Hone)}+\frac{1}{4 \eps } \bar{b}\,\calLfraka^6\|\Delta \wregmt\|_{L^2_t(\Ltwo)}^{2},
		\end{aligned}
	\end{equation}
allowing to absorb the first term on the right by reducing $\eps$. Employing the derived estimates in \eqref{energy_ineq_1} leads to  \eqref{step 3: energy_ineq}.
	\section{Construction of the function $\rho$} \label{appendix: density results}
We present here the construction of the function $\rho$ used in the proof of Lemma~\ref{lemma: Compatible reg data}.
	\begin{lemma} \label{lemma: constructing r}
		Let $0<r<R$ for $R>0$. There exists $\rho \in C^\infty_0(\R)$, such that 
		\begin{equation} \label{values}
			\begin{aligned}
	\rho (\tau)=1 \ \text{ for}\ |\tau|\leq r, \quad
				\rho (\tau)\in [0,1] \ \text{ for}\ r\leq |\tau|\leq R, \quad 
				\rho (\tau)=0\ \text{ for}\ |\tau|\geq R.
			\end{aligned}
		\end{equation}
	\end{lemma}
	\begin{proof}
		Such a function can be constructed as follows (see~\cite{kumaresan2002course}).
		We first consider the smooth $C^\infty (\R)$ transition function 
		\begin{equation}
			f(\tau)=\left\{
			\begin{aligned}
				&e^{-1/\tau},&\quad \text{for}\quad \tau>0\\
				&0,&\quad\text{for}\quad \tau\leq 0. 
			\end{aligned}
			\right. 
		\end{equation}
		Then by induction
		\begin{equation}\label{f_k}
			\partial_\tau^\ell f(\tau)=\left\{
			\begin{aligned}
				&p_\ell(1/\tau)e^{-1/\tau},&\quad \text{for}\quad \tau>0\\
				&0,&\quad\text{for}\quad \tau \leq  0
			\end{aligned}
			\right. 
		\end{equation}				
for a certain polynomial $p_\ell$, $\ell=0,1,2,\dots$  of degree at most $\ell+1$. Since for any polynomial we have 
		
		\[ \displaystyle \lim_{\tau \rightarrow 0}p(\tau)e^{-1/\tau}=0, 
		\]
		we conclude that	 $\partial_\tau^\ell f$ is differentiable at $0$. 
		We next set
		\begin{equation}
			f_1(\tau)=f(\tau-r)f(R-\tau).
		\end{equation}
		It is clear that $f_1\in C^\infty_0(\R)$ since $\mathrm{supp}(f_1) \subset [r, R]$. Now, let  
		\begin{equation}
			f_2(\tau)=\int_\tau^\infty f_1(s)\ds. 
		\end{equation}
		We then have 
		\begin{equation}
			f_2(\tau)=\left\{
			\begin{aligned}
				&0,&&\ \text{for}\ \tau\geq R\\
				\int_r^R & f_1(s)\ds&&\ \text{for}\ \tau\leq r. 
			\end{aligned}
			\right.
		\end{equation}
		Hence, the function 
		\begin{equation}
			\rho(\tau)=\frac{f_2(|\tau|)}{\int_r^R f_1(s)\ds}
		\end{equation}
		belongs to $ C_0^\infty(\R)$ and satisfies \eqref{values}.
Integrating by parts, we obtain
			\begin{equation}
				\int_0^\infty (s \rho(s)+ (1+s^2) \rho'(s) + s \rho''(s)) \ds = - \int_0^\infty s \rho(s) \ds \leq 0.
			\end{equation}
	\end{proof} 
\end{appendices}	
 \bibliography{references}{}   
\bibliographystyle{siam}      
  \end{document}